\documentclass{article}
\usepackage{geometry}
\usepackage{graphicx}
\usepackage{amsmath}
\usepackage{amssymb}
\usepackage{amsthm}
\usepackage{epstopdf}
\usepackage{paralist}
\usepackage{subfig}
\usepackage{listings}
\usepackage{xfrac}

\newcommand{\field}[1]{\mathbb{#1}}
\newcommand {\N}        {\field{N} }
\newcommand {\Z}        {\field{Z} }
\newcommand {\R}        {\field{R} }

\newtheoremstyle{dotless}{}{}{\itshape}{}{\bfseries}{}{ }{}
\theoremstyle{dotless}
\renewenvironment{proof}{{\it Proof }}{\qed}

\theoremstyle{definition}
\newtheorem{thm}{Theorem}[section]
\newtheorem{lm}[thm]{Lemma}
\newtheorem{prop}[thm]{Proposition}
\newtheorem{cor}[thm]{Corollary}
\newtheorem{defn}{Definition}[section]
\newtheorem{ex}{Example}[section]

\title{Cellular Probabilistic Automata - A Novel Method for Uncertainty Propagation}
\author{Dominic Kohler\footnotemark[1]
\and Johannes M\"uller\footnotemark[2] 
\and Utz Wever\footnotemark[3]}

\begin{document}
\maketitle
\newcommand{\slugmaster}{%
\slugger{juq}{xxxx}{xx}{x}{x--x}}
\renewcommand{\thefootnote}{\fnsymbol{footnote}}
\footnotetext[1]{Technical University Munich, Centre for Mathematical Sciences, Boltzmannstr. 3, D-85748 Garching/Munich, Germany; Siemens AG, Corporate Technology, Otto-Hahn-Ring 6, 81730 Munich, Germany (dominic.kohler@mytum.de)}
\footnotetext[2]{Technical University Munich, Centre for Mathematical Sciences, Boltzmannstr. 3, D-85748 Garching/Munich, Germany (johannes.mueller@mytum.de)}
\footnotetext[3]{Siemens AG, Corporate Technology, Otto-Hahn-Ring 6, 81730 Munich, Germany (utz.wever@siemens.com)}
\renewcommand{\thefootnote}{\arabic{footnote}}

\begin{abstract} 
We propose a novel density based numerical method for uncertainty propagation under certain partial differential equation dynamics. The main idea is to translate them into objects that we call cellular probabilistic automata and to evolve the latter. The translation is achieved by state discretization as in set oriented numerics and the use of the locality concept from cellular automata theory. We develop the method at the example of initial value uncertainties under deterministic dynamics and prove a consistency result. As an application we discuss arsenate transportation and adsorption in drinking water pipes and compare our results to Monte Carlo computations. 
\end{abstract}



\section{Introduction}


The numerical treatment of differential equations that are subject to uncertain data has attracted a lot of interest lately. A prominent approach is to use Polynomial Chaos expansions \cite{WienerPC, GhanemSpanos, SchwabPC}. It can be improved by decomposing the random space \cite{WanKarniadakis}, and only recently numerical implementations of this improvement have been investigated \cite{AugustinRentrop}. Alternative approaches are based on the Monte Carlo idea, like the Markov Chain Monte Carlo Method \cite{SpiegelhalterGilksMCMC}, Latin Hypercube Sampling \cite{LohLHS}, the Quasi Monte Carlo Method \cite{FoxQMC}, importance sampling \cite{MelchiersIS} and the Multi-Level Monte Carlo Method \cite{SchwabMLMC}. Further well-known approaches use the It\^o calculus \cite{Oksendal,KloedenPlaten, KloedenPlaten2, Roessler} or the Fokker-Planck equation \cite{Langtangen}. Although the approaches have proven to be successful for many tasks, they often encounter certain efficiency restrictions in higher dimensions of the random space. New methods are needed to meet these challenges.

Time-continuous dynamical systems on continuous state space can be approximated by time-discrete Markov chains on finite state space \cite{Hsu}. This technique of state space discretization has led to the powerful tools of set oriented numerics \cite{DellnitzJunge1, DellnitzJunge2}. It is especially useful to study ergodic theory, asymptotic dynamics and optimal control \cite{Kushner, GrueneJunge}. Recently, also contributions to uncertainty quantification have been made \cite{JungeMarsdenMezic}. 

In this paper we introduce a novel numerical scheme for uncertainty propagation in certain spatio-temporal processes. It is based on the concept of state space discretization and on ideas from cellular automata (CA) theory \cite{CAsurveyKari, CApatternformation, stochasticCA}. We develop the method at the example of the propagation of initial value uncertainties under deterministic partial differential equation (PDE) dynamics and pave the way towards an extension to more general stochastic influences on the system. 

In particular we introduce a discretization of PDE which do not depend explicitly on the independent variables. First a finite difference scheme is applied to a PDE; the spatial and temporal continuum is replaced by discrete sites and discrete time steps. Second the state of the resulting system is discretized. Because this procedure emphasizes the interaction between neighboring sites, a property that strongly resembles the locality and shift-invariance in cellular automata, the resulting completely discrete system is termed a cellular probabilistic automaton (CPA). Such an automaton is much simpler than the PDE and becomes accessible to very efficient simulation techniques. 

CPA basically consist of information about transition probabilities between discretized portions of phase space in a site's neighborhood. The transition probabilities are interpreted as to approximate the evolution of the system's probability density in transfer operator theory \cite{Lasota}. Hence CPA can be used for uncertainty propagation \cite{MezicRun}. While the translation from PDE into CPA may be rather time-consuming, the evolution of uncertainties with CPA is fast. The accuracy of the approximation depends on two parameters: one measures the state space resolution at every site, and the other the degree of locality, i.e. the extent to which correlations between neighboring sites are preserved. 

The paper is structured as follows. In Section \ref{sec:Densityup} we formulate the problem of initial value uncertainty propagation under deterministic dynamics and deterministic boundary conditions. Here we also present the idea of density based uncertainty propagation through phase space discretization. By exploiting locality and shift-invariance of our problem this leads to the definition and discussion of CPA in Section \ref{sec:CPA}. A consistency result for our construction is presented in Section \ref{sec:Consistency}. In Section \ref{sec:Example} we show how CPA can be extended to incorporate stochastic boundary conditions and apply the theory to the example of arsenate transportation and adsorption in water pipes. The results are compared to Monte Carlo computations. Finally we conclude in Section \ref{sec:Conclusions}.

\section{Density Based Uncertainty Propagation}

\label{sec:Densityup}

In this section we first formulate the problem and then develop the idea of density based uncertainty propagation. Finally the CPA idea is derived in this context. 

\subsection{Problem Formulation}

We are interested in the time evolution of uncertain initial data in a specific deterministic dynamical system. First we introduce some notation from probability theory and the Frobenius-Perron operator as the suitable tool to describe this process. Second we specify the deterministic dynamical system that we will work with, and third we formulate the problem. 

Let $(X, \mathcal{A}, \mu)$ be a probability space, $(X', \mathcal{A}', \mu')$ a measure space and $V:X \to X'$ a random variable with distribution $\mu_V$. We say that $V$ has density $g$ if there is $g \in \mathcal{L}^1(X',\mathcal{A}', \mu')$ such that 
$$\mu_V[A']=\int_{A'} g d\mu', \quad \text{ for all } A' \in \mathcal{A}'.$$
The set of densities on $(X',\mathcal{A}',\mu')$ is denoted by
$$D(X'):=\{g \in \mathcal{L}^1(X',\mathcal{A}',\mu')\,|\, g\geq 0, \|g\|_1=1\}.$$
Let now $(X', \mathcal{A}', \mu')=(\R^{mn},\mathcal{B}(\R^{mn}),\lambda)$, where $m,n \in \N$, $\mathcal{B}(\R^{mn})$ is the Borel $\sigma$-algebra and $\lambda$ the Lebesgue measure. A measurable map $S:\R^{mn} \to \R^{mn}$ is called non-singular if $\lambda(S^{-1}(A'))=0$ for all $A' \in \mathcal{B}(\R^{mn})$ with $\lambda(A')=0$. For any such map a unique operator can be defined on the basis of the Radon-Nikodym theorem \cite{Lasota}. 

\begin{defn}
Given a non-singular map $S:\R^{mn} \mapsto \R^{mn}$, for $g \in \mathcal{L}^1(\R^{mn})$ the Frobenius-Perron operator (FPO) $P_S: \mathcal{L}^1(\R^{mn}) \to \mathcal{L}^1(\R^{mn})$ is defined by 
$$\int_{A'} P_Sg(x)dx=\int_{S^{-1}(A')}g(x)dx \quad \forall A' \in \mathcal{B}(\R^{mn}).$$
\end{defn}

The FPO preserves positivity and normalization and hence describes how densities are mapped under phase space evolution. We focus on a particular type of phase space evolution. 

\begin{defn}
Consider a deterministic dynamical system $(T,\R^{mn},\Phi)$ specified as follows: 
\begin{compactitem}
\item [i)] $(T,+)$ is an additive half-group of time, 
\item [ii)] $\R^{mn}=\times_{I}\R^n$ is the state space, where $I=\{1,...,m\}$, 
\item [iii)] the flow $\Phi: T \times \R^{mn} \to \R^{mn}$ is non-singular for all $t \in T$,
\item [iv)] there is a neighborhood $U=\{-r,...,s\}$ with $r,s \in \N_0$, $r+s\leq m$, such that $\Phi$ has the locality property, i.e. that there is $h: T \times (\R^n)^{|U|} \to \R^n$ with 
$$\Phi(t,v)_i=h(t, v_{i-r},...,v_{i+s})$$
for all $t \in T, v=(v_1,...,v_m) \in \times_{I}\R^n$ and $i \in \{1+r,...,m-s\}$, 
\item [v)] and that the system acts as the identity on $K=\{1,...,r\}\cup \{m-s+1,...,m\}$, i.e. $\Phi(t,v)|_K=v|_K$ for all $t \in T$ and all $v \in \times_{I}\R^n$. 
\end{compactitem}
\end{defn}

We will write $\Phi^t(v):=\Phi(t,v)$ in the following and refer to \cite{GuckenheimerHolmes} for further information on dynamical systems. Assume that there is a compact $\Omega \subsetneq \R^n$ such that $\Omega^m$ is positively invariant under the flow and fix $\tau \in T, \tau \neq 0$. 

Our main application is the analysis of a PDE 
$$\partial_tv=\tilde h(\partial_{xx}v, \partial_x v, v), \quad v(x,t)\in \Omega$$
on a one-dimensional compact spatial domain $x \in [a,b]$ for $a,b \in \R$. Under certain regularity assumptions a dynamical system like the above is obtained by applying a finite difference method with space discretization $\Delta x=\frac{b-a}{m-1}$, where $m \in \N, m\geq 2$, and time step $\tau$. Then $U$ is naturally induced by the choice of the finite difference scheme; e.g. usually $U=\{-1,0,1\}$ is suitable to account for central second order difference quotients. Because of the PDE context we call $I$ the set of sites. By only considering trajectories with $v^0|_K=k \in \times_K\R^n$, the system can be interpreted as to obey boundary conditions. 

The time evolution of uncertain initial data in the deterministic dynamical system is described by real random variables $V^0,V^1,...:X \to \Omega^{m}$ on probability space $(X,\mathcal{A},\mu)$, where $V^{n+1}=\Phi^\tau V^n$. We focus on deterministic boundary conditions: $V^0(x)|_K=k \in \times_K\R^n$ for all $x \in X$. If $V^n$ has density $g^n \in D(\R^{mn})$, the density of $V^{n+1}$ is given by application of the associated FPO: $g^{n+1}=P_{\Phi^\tau}(g^n)$. The goal is to develop an algorithm that approximates the density evolution. It will be achieved by translating the system into a CPA in two steps. 
First the FPO is discretized via a state discretization procedure, and then locality and shift-invariance are used to further transform it into a CPA.

\subsection{State Space Discretization}

In this section first we introduce the concept of state space discretization. Second we investigate according densities, and third we construct a discretized version of the FPO. In principle these ideas are well-known in the literature \cite{DellnitzJunge1, DellnitzJunge2}. Here they are adapted to the special structure of the dynamical system. 

\begin{defn}
A partition or coding $E$ of $\Omega$ is a finite collection of disjoint sets $\{\Omega_e\}_{e \in E}$ whose union is $\Omega$. We call $e \in E$ the symbol of coding domain $\Omega_e$, and the coding map is the function $T: \Omega \mapsto E$, where $T(v)=e$ if $v \in \Omega_e$. A partition is called uniform, if there is a resolution $\Delta \Omega \in \R$ such that $\Omega_{e}$ is a $n$-dimensional hypercube with side length $\Delta \Omega$ for all $e \in E$. 
\end{defn}

To avoid technical complications in the following proofs we only consider uniform partitions while developing the theory. They are also the ones that are relevant in practical algorithms. 

A partition $E$ of $\Omega$ with coding map $T$ and $|E|=N$ naturally induces a partition $E^I$ of $\Omega^m$ with coding map 
$$\hat{T}: \Omega^m \to E^I, v \mapsto \hat{T}(v) \text{ with }(\hat{T}(v))_i=T(v_i) \text{ for } i \in I.$$ 
Note that $|E^I|=mN$. For $\varphi \in E^J$, where $J\subseteq I$, we write 
$$\Omega_{\varphi}=\{v \in \Omega^{m}\,|\, \forall j \in J:\, \hat T(v)(j)=\varphi (j) \}.$$

Now we study densities that are compatible with state space discretization. For this purpose we introduce the measure space $(E^I,\mathcal{P}(E^I), \gamma)$, where $\mathcal{P}(E^I)$ is the power set of $E^I$ and $\gamma$ is the counting measure. 
The densities $D(E^I)$ consist of the weight functions 
$$g:E^I \to [0,\infty], \quad g(\varphi)=p_\varphi,$$
where $(p_\varphi)_{\varphi \in E^I}$ are nonnegative numbers with $\sum_{\varphi \in E^I}p_\varphi=1$. 

\begin{defn} $ $
\begin{compactitem}
\item [i)] $\mathcal{L}^1_{\hat T}(\R^{mn})=\text{span}(B)$ is the finite-dimensional $\mathcal{L}^1(\R^{mn})$-subspace of piecewise constant functions with basis $B=\{\chi_{\Omega_\varphi}/\lambda(\Omega_\varphi)\}_{\varphi \in E^I}$. The set of piecewise constant densities is given by
$D_{\hat T}(\R^{mn}):=\mathcal{L}^1_{\hat T}(\R^{mn})\cap D(\R^{mn}).$
\item [ii)] The coordinate representation $\kappa_B: \mathcal{L}^1_{\hat T}(\R^{mn}) \to \R^{E^I}, \quad g \mapsto \kappa_B(g)$ with respect to the basis $B$ is given by 
$\kappa_B(g)(\varphi)=c_\varphi$ for $\varphi \in E^I$ and $g=\sum_{\psi \in E^I}\frac{c_\psi}{\lambda(\Omega_\psi)}\chi_{\Omega_{\psi}}$. Obviously $\kappa_B(D_{\hat T}(\R^{mn}))=D(E^I)$. 
\item [iii)] Let $\rho\in E^K$ such that $\rho_i=T(k_i)$ for all $i \in K$. The densities that are compatible with the boundary conditions are given by 
$$D_{BC}(E^I):=\{g \in D(E^I) \,|\, g(\varphi)=0 \text{ if } \varphi |_K\neq \rho \}.$$ 
\end{compactitem}
\end{defn}

By averaging in the coding domains every function in $\mathcal{L}^1(\R^{mn})$ can be mapped to a piecewise constant function. 

\begin{defn}
A restriction operator to the subspace of piecewise constant functions is given by  
$$R: \mathcal{L}^1(\R^{mn}) \to \mathcal{L}^1_{\hat T}(\R^{mn}), \quad R(g)=\sum_{\varphi \in E^I}\frac{c_\varphi}{\lambda(\Omega_{\varphi})} \chi_{\Omega_{\varphi}},$$
where 
$$c_\varphi=\int_{\Omega_{\varphi}}g(w)dw.$$
\end{defn}

$R$ is idempotent, i.e. $R \circ R=R$, and furthermore $R(D_{\hat T}(\R^{mn}))\subseteq D_{\hat T}(\R^{mn})$. In the following we will use the restriction operator to construct a discretized version of the FPO on density level: $RP_{\Phi^\tau}$. This procedure is well-known in ergodicity theory when invariant measures are approximated. There it is called Ulam's method \cite{Ulam}. 

The matrix representation of the linear $RP_{\Phi^\tau}|_{\mathcal{L}^1_{\hat T}(\R^{mn})}$ is given by $P_B=\kappa_BRP_{\Phi^\tau}\kappa_B^{-1} \in \R^{E^I\times E^I}$ with entries
\begin{align*}
P_{B,\varphi,\psi}=\int_{\Omega_{\psi}} P_{\Phi^\tau}\frac{\chi_{\Omega_\varphi}}{\lambda(\Omega_\varphi)}d\lambda=\int_{\Phi^{-\tau}(\Omega_\psi)} \frac{\chi_{\Omega_\varphi}}{\lambda(\Omega_\varphi)}d\lambda=\frac{\lambda(\Omega_\varphi \cap \Phi^{-\tau}(\Omega_\psi))}{\lambda(\Omega_\varphi)}.
\end{align*}
$P_{B,\varphi,\psi}$ is the probability of finding a realization of a random variable with uniform density in $\Omega_\varphi$ in $\Omega_\psi$, when $\Phi^\tau$ is applied. Hence we may interpret $P_{B,\varphi,\psi}$ as the transition rate from $\Omega_\varphi$ to $\Omega_\psi$ of a finite state Markov chain on $\{\Omega_\varphi\}_{\varphi \in E^I}$. This chain approximates the behavior of the dynamical system for uncertain initial values. 

In the following we regard $P_B:D_{BC}(E^I)\to D_{BC}(E^I)$ as a function which maps densities that are compatible with the boundary conditions by matrix multiplication.

\subsection{Using Locality - Towards Cellular Probabilistic Automata}

$E^I$ grows exponentially in $m$. For a growing number of sites it becomes numerically expensive to obtain global transition rates and to handle global states and densities. 

However, our dynamical system has a special structure: We use the locality property to approximate the set of global transition probabilities by several identical sets of local ones. This is possible out of two reasons. First because we find identical dynamics at all sites away from the boundaries, and second because the transition probabilities at one particular site mainly depend on the state of its neighborhood rather than on the whole global configuration. \\

For the formal definition of these local transition probabilities we need to introduce the shift by $l \in \Z$ on finite grid $J \subset \Z$. It is given by
$$\sigma_l: F^J \to F^{-l+J}, \quad \varphi \mapsto \sigma_l(\varphi), \quad \sigma_l(\varphi)(-l+j)=\varphi(j),$$
where $F$ is an arbitrary set, e. g. $F=E$ or $F=D(E^V)$. Moreover, for arbitrary $V=\{-p,...,q\}, W=\{-t,...,u\}$ with $p,q,t,u \in \N_0$ and $l \in \Z$ we use the conventions $l+V=\{-p+l,...,q+l\}$ and $V+W=\{-p-t,...,q+u\}$.

\begin{defn} 
\label{def:cpafromode}
Let $V=\{-p,...,q\}$ with $p,q \in \N_0$ and $p+q+r+s\leq m$. A local function $f_0: E^{U+V} \to D(E^V)$ is then given by 
$$f_0(\varphi)(\psi)=\frac{\lambda (\Omega_{\sigma_{-i}(\varphi)} \cap \Phi^{-\tau}(\Omega_{\sigma_{-i}(\psi)}))}{\lambda (\Omega_{\sigma_{-i}(\varphi)})},$$
where $i=1+p+r$, $\varphi \in E^{U+V}$ and $\psi \in E^V$ 
\end{defn}

Note that because of the locality property the definition is independent of the chosen site $i \in \{1+p+r,...,m-q-s\}$. The set $V$ controls the degree of locality, i.e. the number of sites that give rise to a local transition. It will turn out that by enlarging it we can diminish the error of the locality approximation. 

In the following section we develop a method of how to combine several such local transitions to approximate a global one. This will finish the construction of a CPA from the FPO.

\section{Cellular Probabilistic Automata}

\label{sec:CPA}

CPA are defined by extending the definition of deterministic CA according to \cite{CApatternformation, stochasticCA}: In  CPA the local transition function specifies a time- and space-independent probability distribution of next states for each possible neighborhood configuration. Because we do not want to follow one realization but rather the whole ensemble, unlike in the literature we define CPA to work on densities. This enables their utilization for uncertainty propagation. 

In the last chapter we showed how the discretized FPO $P_B$ on state space $D_{BC}(E^I)$ can be used to approximate the FPO $P_{\Phi^\tau}$ on $D(\R^{mn})$. CPA further approximate the discretized FPO on a product space of local densities, see Fig. \ref{fig:relations} for a sketch. Uncertainty propagation with CPA therefore requires two definitions. First one about how to translate between global densities and the product space of local densities, and second one about how to evolve local densities in time with the help of the local function. 

Because the definitions can be best understood for $V=\{0\}$, in Section \ref{subsec:specialCPA} we first introduce CPA in this special case to demonstrate the basic construction. Afterwards we develop the de Bruijn calculus as a connection between local and global objects for more general $V$ in Section \ref{subsec:deBruijn}. This connection leads to the generalization of CPA to general $V$ in Section \ref{subsec:generalCPA}. 

\begin{figure}
\centering
\includegraphics[width=15cm]{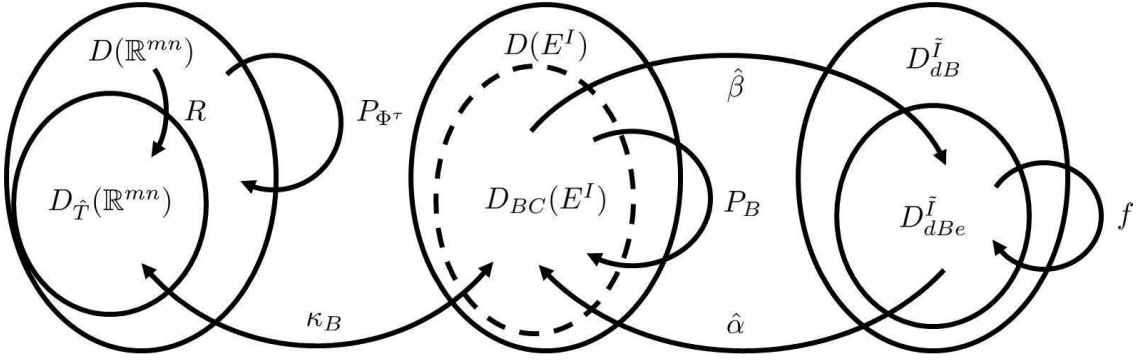}
\caption{The relations between the FPO $P_{\Phi^\tau}$ and its approximations. By state discretization we obtain the discretized FPO $P_B$ which still works globally, and by exploiting locality we approximate $P_B$ further by the CPA with global function $f$. The state space on which the CPA operates is a collection of local densities, see text. }
\label{fig:relations}
\end{figure}

\subsection{Cellular Probabilistic Automata: A Special Case}

\label{subsec:specialCPA}

A crucial step is to translate between global densities $D_{BC}(E^I)$ and a (subset of a) collection of local densities $(D(E^V))^{\tilde I}$, where $\tilde I$ contains the sites away from the boundary. We introduce an operator $\hat \beta: D_{BC}(E^I) \to (D(E^V))^{\tilde I}$ and a concatenation operator $\hat \alpha: (D(E^V))^{\tilde I} \to D_{BC}(E^I)$. $\hat \beta$ localizes the information to densities on states of length $V$ and thus erases far-reaching correlations. $\hat \alpha$ in turn constructs global densities out of information about local densities. As we will see in the next section, this process is by no means unique and requires some technical refinement of the space of local densities. However, for $V=\{0\}$ there are canonical definitions for $\hat \alpha$ and $\hat \beta$: multiplication of local probabilities for independent events and calculation of marginal distributions. 

\begin{defn} 
As before let $I=\{1,...,m\}$, $U=\{-r,...,s\}$ for $r,s,m \in \N_0$ with $1 \leq m, 1+r+s\leq m, $ and $K=\{1,...,r\}\cup \{m-s+1,...,m\}$. Furthermore $\tilde I=\{1+r,...,m-s\}$. 
\begin{compactitem}
\item [i)] We set $\hat \alpha: (D(E))^{\tilde I} \to D_{BC}(E^I), \,g \mapsto \hat \alpha(g)$ with
\begin{align*}
\hat \alpha(g)(\psi) &= \left\{\begin{array}{ll}
\prod_{i \in \tilde I}g(i)(\psi(i)) & \mbox{ if  } \psi |_K=\rho\\
0 & \mbox{ else } \\
\end{array}\right. 
\end{align*}
\item [ii)] and $\hat \beta: D_{BC}(E^I) \to (D(E))^{\tilde I}, \, g \mapsto \hat \beta (g)$ with 
$$\hat \beta (g)(i)(e)=\sum_{\chi \in E^I \text { s. t. }\chi (i)=e}g(\chi).$$
\end{compactitem}
\end{defn}

\begin{defn}
A cellular probabilistic automaton (CPA) is a tuple $(I, U, E, f_0)$, where for $m,r,s \in \N_0$ with $1 \leq m$ and $1+r+s\leq m$
\begin{compactitem}
\item [i)] $I=\{1,...,m\}$ is a finite grid, 
\item [ii)] $U=\{-r,...,s\}$ is the neighborhood, 
\item [iii)] $E$ is a finite set of local states
\item [iv)] and $f_0 : E^{U} \rightarrow D(E)$ is the local function. 
\end{compactitem}

With the boundary conditions $\rho \in E^K$ the global function is given by
\begin{equation*}
f: (D(E))^{\tilde I} \rightarrow (D(E))^{\tilde I}, \quad g \mapsto f(g), 
\end{equation*}
$$f(g)(i)(\psi) = \sum_{\overset{\varphi \in E^{U} \text{ s.t. } \varphi(k-i)=\rho(k)}{\text{for }k \in K\cap i+U}} \prod_{j \in (i+U)\backslash K} g(j)(\varphi(j))f_0(\varphi)(\psi).$$
The trajectory starting with  $g^0 \in (D(E))^{\tilde I}$ is given by the sequence $(g^n)_{n \in \N}$, where $g^n=f(g^{n-1})$ for $n \in \N ^+$.
\end{defn}

A CPA can be used to evolve an input distribution $\hat \beta (g)$ for $g \in D_{BC}(E^I)$ via the global function. After $n$ time steps the approximated global density is then given by $\hat \alpha f^n \hat \beta (g)$, see also Fig. \ref{fig:relations} with $D_{dB}^{\tilde I}=D_{dBe}^{\tilde I}=(D(E))^{\tilde I}$. The role model for the global function is the matrix operation with the discretized FPO $P_B$: The product of the transition probability with the probability of being in a preimage state is summed up over all possible preimage states. A probability is assigned to a preimage state $\varphi \in E^U$ by multiplication of local probabilities like in the definition of $\hat \alpha$. 

To cope with boundary conditions it is necessary that the global function only operates on $\tilde I$ instead of $I$. Deterministic CA are special cases of CPA: assume that for all $\varphi \in E^U$ there is $e \in E$ such that $f_0(\varphi)(e)=1$ and that the input is deterministic.

\subsection{De Bruijn Calculus}
\label{subsec:deBruijn} 

To generalize the construction to arbitrary $V$ we first study the relation between local and global objects in more depth. We introduce the de Bruijn density calculus on the basis of pattern ideas in cellular automata theory \cite{MSSCA, VollmarCA}, in the theory of de Bruijn graphs \cite{sutner} and in pair approximation \cite{pairformation}. 

As before, we introduce an operator $\beta_W$ that localizes the global information to densities on states of length $V$, this time $|V|\geq 1$. The precise definition of $\beta_W$ is still rather straight forward: marginal distributions dismiss all information but the one over a certain range $V$. We will find below that the reconstruction of global densities out of local information by $\alpha_W$ is more involving. However, let us first define $\beta_W$. 

\begin{defn}
Let $V=\{-p,...,q\}$ and $W=\{-t,...,u\}$ for $p,q \in \N$ and $t,u \in \Z, -t \leq u$. $\beta_W: D(E^{V+W}) \to (D(E^V))^W$ is given by
\begin{align*}
\beta_W(g)(i)(\psi)&=\sum_{\overset{\chi \in E^{V+W} \text{ s. t. }}{\chi|_{i+V}=\sigma_{-i}(\psi)}}g(\chi).
\end{align*}
\end{defn}

\begin{ex}
\label{ex:beta}
Let $E=V=W=\{0,1\}$, $p \in (0,1)$ and $g,\tilde g \in D(E^{V+W})$ be given by 
\begin{align*}
g(\psi) &= \left\{\begin{array}{ll}
p & \mbox{ if  } \psi=(001)\\
1-p & \mbox{ if  } \psi=(100)\\
0 & \mbox{ else } \\
\end{array}\right., &\tilde g(\psi)=\left\{\begin{array}{ll}
p(1-p) & \mbox{ if  } \psi=(000), \psi =(101)\\
p^2 & \mbox{ if  } \psi=(001)\\
(1-p)^2 & \mbox{ if  } \psi=(100)\\
0 & \mbox{ else } \\
\end{array}\right..
\end{align*}
We find that $\beta_W(g)=\beta_W(\tilde g)$ with 
\begin{align*}
\beta_W(g)(0)(\psi) &= \left\{\begin{array}{ll}
p & \mbox{ if  } \psi=(00)\\
1-p & \mbox{ if  } \psi=(10)\\
0 & \mbox{ else } \\
\end{array}\right., &\beta_W(g)(1)(\psi)=\left\{\begin{array}{ll}
(1-p) & \mbox{ if  } \psi=(00)\\
p & \mbox{ if  } \psi=(01)\\
0 & \mbox{ else } \\
\end{array}\right..
\end{align*}
\end{ex}
Ex. \ref{ex:beta} shows that information is lost under $\beta_W$, i.e. different global densities are mapped to the same collection of local densities. Now we are interested in $\alpha_W: D(E^{V+W}) \to (D(E^V))^W$. Although the properties of $\beta_W$ allow to define $\alpha_W$ as the solution of a linear nonnegative least squares problem \cite{nnls}, this algebraic approach is not appropriate. We rather suggest a probabilistic approach that fulfills two requirements: first $\alpha_W$ shall degenerate to simple multiplication of local densities for $V=\{0\}$, and second $\alpha_W$ and $\beta_W$ shall be inverse on important sets. 
For this purpose we first introduce several definitions, see also Fig. \ref{fig:alphabeta}. 

\begin{defn} $ $ 
\begin{compactitem}
\item [i)] $X_{dB}^W=(\mathcal{P}(E^V))^W$ is the set of de Bruijn states, where $\mathcal{P}(E^V)$ is the power set of $E^V$. The elements of $E^V$ are called patterns of size $V$.
\item [ii)] The subset of extendable de Bruijn states is given by $$X_{dBe}^W=\{\Phi \in X_{dB} \,|\, \forall i \in W \,\forall \varphi \in \Phi(i) \, \exists \psi \in E^{V+W} \, \forall i \in W : \,\sigma_i(\psi)|_{V}=\varphi \}, $$ the set of completely extendable de Bruijn states. 
\item [iii)] $D^W_{dB}=(D(E^V))^W$ is called the set of de Bruijn densities. 
\item [iv)] The subset of extendable de Bruijn densities is given by $$D^W_{dBe}=\{g \in D^W_{dB}\, |\, \times_{i \in W} \text{supp } g(i) \in X^W_{dBe}\}.$$ 
\end{compactitem}
\end{defn}

\begin{figure}
\centering
\hfill
\subfloat[\label{fig:alphabeta}]{\includegraphics[width=10cm]{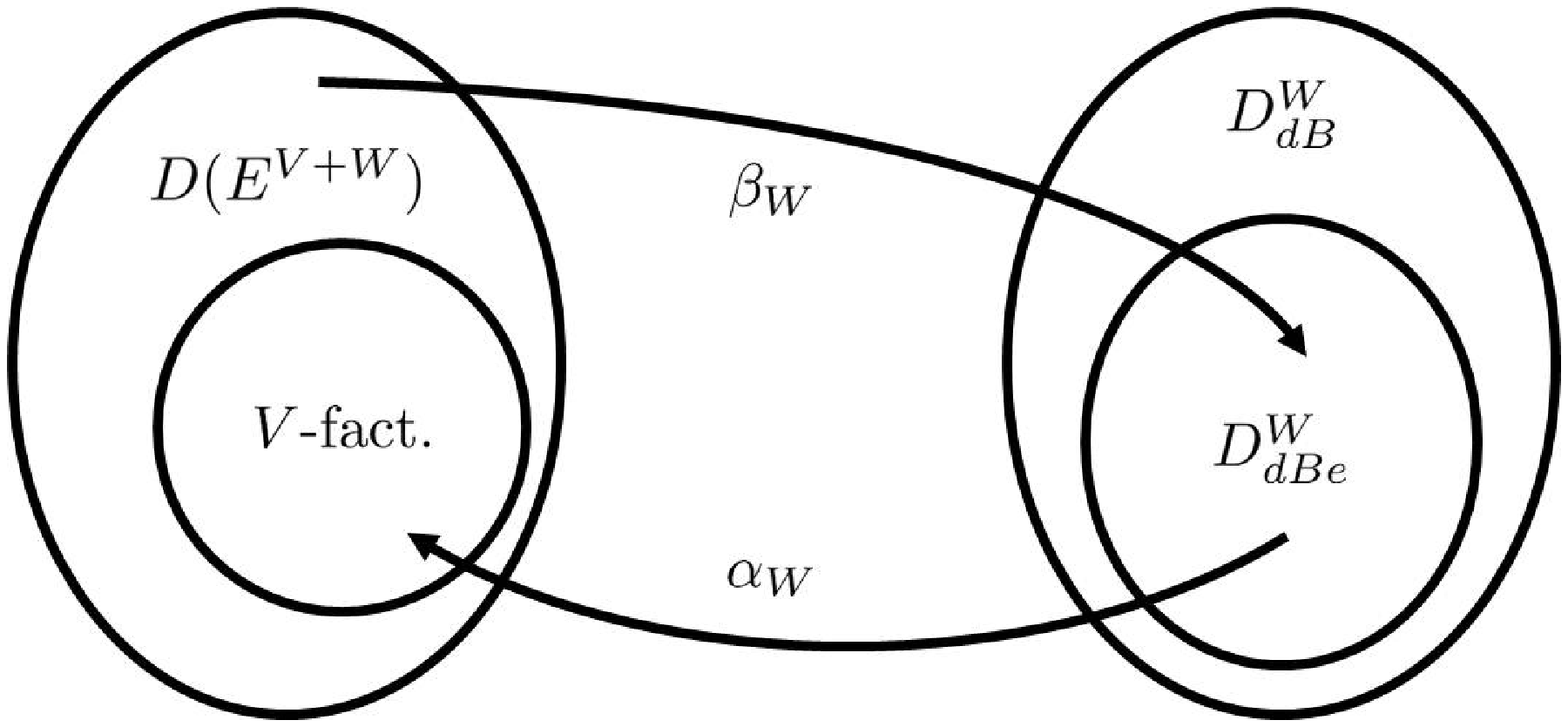}}
\hfill 
\subfloat[\label{fig:approximation}]{\includegraphics[width=4.5cm]{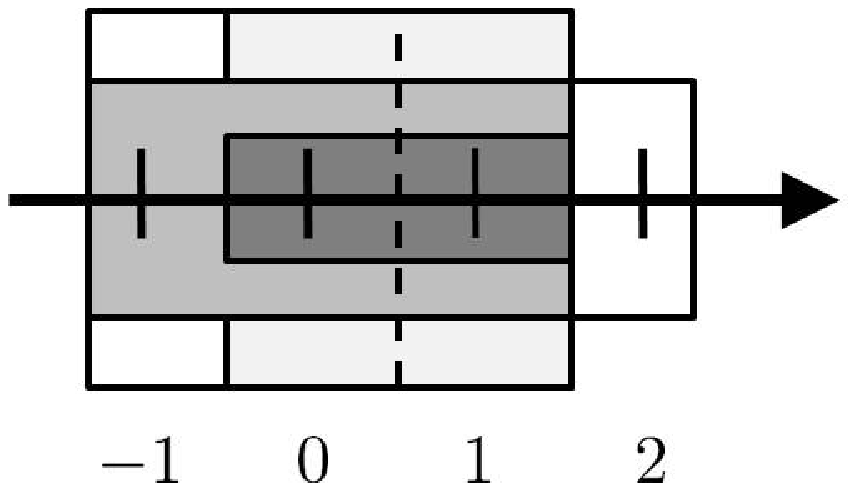}}
\hfill
\caption{(a) The relation between global densities and de Bruijn densities. (b) An example of an approximation for $W=\{-1,0,1\}$, $V=\{0,1\}$ and $i=0$. The thin and dark grey box that covers sites $0$ and $1$ is the factor at site $0$. The medium box that covers sites $-1$ to $2$ is the factor at site $1$: the local state at site $2$ depends on the local states from $-1$ to $1$ (medium grey part). In the approximation (dashed line) the local state at $2$ only depends on the one at $1$. The thick box covering sites $-1$ to $1$ is the factor at site $-1$. The local state at site $-1$ depends on the local states on sites $0$ and $1$ (light grey part). In the approximation the dependence stops again at the dashed line.}
\label{fig:deBruijn}
\end{figure}

The idea behind extendable de Bruijn states is that every pattern can be extended to a global state by gluing suitable patterns on it. An example of an extendable de Bruijn density is $\beta_W(g)$ in Example \ref{ex:beta}: for example pattern $(10)$ at site $0$ can be extended by $(01)$ at site $1$ to the global state $(101)$, because the patterns coincide in the overlapping state $0$. We find that this observation can be generalized. 

\begin{lm}
\label{lm:betadbce}
$\text{im }(\beta_W)\subseteq D^W_{dBe}$. 
\end{lm}

\begin{proof}
Let $g\in \text{im } (\beta_W)$, $j \in W$ and $\varphi \in E^V$ with $g(j)(\varphi)>0$. Then there is $\tilde g \in D(E^{V+W})$ and $\psi \in E^{V+W}$ such that $\tilde g(\psi)>0$ and $\sigma_j(\psi)|_{V}=\varphi$. But furthermore already $g(i)(\sigma_i(\psi)|_{V})>0$ for all $i \in W$, and therefore $\times_{i \in W} \text{supp } g(i) \in X^W_{dBe}$. 
\end{proof} 


Because only the extendable de Bruijn densities are addressed by $\beta_W$, we only look for $\alpha_W$ on them. Our choice is motivated by the following calculation, for which we first introduce some notation. For $J \subseteq \Z$ and $A,B \subseteq E^{J}$, $\mu[A]=\sum_{\varphi \in A} g(\varphi)$ denotes the distribution associated with $g \in D(E^{J})$, and $\mu[A|B]=\frac{\mu[A\cap B]}{\mu[B]}$ the conditional probability. Furthermore
$$\left\{\psi|^J_{\tilde J}\right\}=\left\{\varphi \in E^{J} \,|\, \varphi|_{\tilde J}=\psi|_{\tilde J}\right\}$$
for $\tilde J, \hat J \subseteq \Z, \tilde J \subseteq J$, $\tilde J \subseteq \hat J, $ and $\psi \in E^{\hat J}$, and we also write $\left\{\psi|_{\tilde J}\right\}=\left\{\psi|^J_{\tilde J}\right\}$ if $J$ is clear from the context. We calculate for $i \in W$ and $\psi \in E^{V+W}$ that
\begin{align*}
\mu[\{\psi\}]=& \mu[\{\psi|_{\{-t,...,i\}+V}\}]\, \prod_{l=i+1}^u \frac{\mu[\{\psi|_{\{-t,...,l\}+V}\}]}{\mu[\{\psi|_{\{-t,...,l\}+V_-}\}]} \\
=& \prod_{k=-t}^{i-1}\mu[\{\psi|_{\{k-p\}}\}\, |\, \{\psi|_{\{k,...,i\}+V_+}\}] \, \mu[\{\psi|_{i+V}\}]\, \prod_{l=i+1}^u\mu[\{\psi|_{\{l+q\}}\}\, |\, \{\psi|_{\{-t,...,l\}+V_-}\}],
\end{align*}
where $V_+=\{-p+1,...,q\}$ and $V_-=\{-p,...,q-1\}$. If there are no far-reaching correlations, we expect the following approximations to be suitable: 
\begin{align*}
\mu[\{\psi|_{\{k-p\}}\}\, |\, \{\psi|_{\{k,...,i\}+V_+}\}] &\approx \mu[\{\psi|_{\{k-p\}}\}\, |\, \{\psi|_{k+V_+}\}], \\
\mu[\{\psi|_{\{l+q\}}\}\, |\, \{\psi|_{\{-t,...,l\}+V_-}\}] &\approx \mu[\{\psi|_{\{l+q\}}\}\, |\, \{\psi|_{l+V_-}\}]
\end{align*}
for $k \in \{-t,...,i-1\}$ and $l \in \{i+1,...,u\}$. They resemble a Markov property of order $|V|-1$ in space, see also Fig. \ref{fig:approximation}: a site's state is independent from the states at sites that are more than $|V|-1$ sites apart. 

\begin{lm}
Let $\mu_j$ be the distribution associated with $\beta_W(g)(j) \in D(E^V)$ for $j \in W$. Then
\begin{align*}
\mu[\{\psi|_{\{k-p\}}\}\, |\, \{\psi|_{k+V_+}\}] &= \mu_k[\{\sigma_k(\psi)|_{\{-p\}}\}\, |\, \{\sigma_k(\psi)|_{V_+}\}], \\
\mu[\{\psi|_{i+V}\}] &= \mu_i[\{\sigma_i(\psi)|_{V}\}], \\
\mu[\{\psi|_{\{l-p\}}\}\, |\, \{\psi|_{l+V_-}\}] &= \mu_l[\{\sigma_l(\psi)|_{\{-p\}}\}\, |\, \{\sigma_l(\psi)|_{V_-}\}], \\
\end{align*}
for $i \in W$, $k \in \{-t,...,i-1\}$ and $l \in \{i+1,...,u\}$. 
\end{lm}

{\em Proof}
Without loss of generality we prove the statement only for $k \in \{-t,...,i-1\}$: 
\begin{align*}
\mu[\{\psi|_{\{k-p\}}\}\, |\, \{\psi|_{k+V_+}\}] &=\frac{\sum_{\chi \in \{\psi|_{k+V}\}}g(\chi)}{\sum_{\chi \in \{\psi|_{k+V_+}\}}g(\chi)} = \frac{\beta_W(g)(k)(\sigma_k(\psi)|_{V})}{\sum_{\varphi \in \left\{\sigma_k(\psi)|^V_{V_+}\right\} }\sum_{\chi \in \left\{\sigma_{-k}(\varphi)|_{k+V}^{V+W}\right\}}g(\chi)} \\ 
&=\frac{\beta_W(g)(k)(\sigma_k(\psi)|_{V})}{\sum_{\varphi \in \left\{\sigma_k(\psi)|^V_{V_+}\right\}}\beta_W(g)(k)(\varphi)} = \mu_k[\{\sigma_k(\psi)|_{\{-p\}}\}\, |\, \{\sigma_k(\psi)|_{V_+}\}]. \qquad \endproof
\end{align*} 

Using the approximations introduced above and the lemma we find that

\begin{align*}
\mu[\{\psi\}]\approx &\prod_{k=-t}^{i-1}\mu[\{\psi|_{\{k-p\}}\}\, |\, \{\psi|_{k+V_+}\}] \, \mu[\{\psi|_{i+V}\}]\, \prod_{l=i+1}^u\mu[\{\psi|_{\{l+q\}}\}\, |\, \{\psi|_{l+V_-}\}] \\
=& \prod_{k=-t}^{i-1}\mu_k[\{\sigma_k(\psi)|_{\{-p\}}\}\, |\, \{\sigma_k(\psi)|_{V_+}\}] \, \mu_i[\{\sigma_i(\psi)|_{V}\}]\, \prod_{l=i+1}^u\mu_l[\{\sigma_l(\psi)|_{\{q\}}\}\, |\, \{\sigma_l(\psi)|_{V_-}\}]. 
\end{align*}

This way we are led to the following definition. 

\begin{defn}
Let $i \in W$. Then $\alpha_{W,i}: D^W_{dBe} \to D(E^{V+W})$ is given by
$$\alpha_{W,i}(g)(\psi)=\prod_{k=-t}^{i-1}\mu_k[\{\sigma_k(\psi)|_{\{-p\}}\}\,|\, \{\sigma_k(\psi)|_{V_+}\}]\, \mu_i[\{\sigma_i(\psi)|_{V}\}]\, \prod_{l=i+1}^u\mu_l[\{\sigma_l(\psi)|_{\{q\}}\}\, |\, \{\sigma_l(\psi)|_{V_-}\}], $$
where $\mu_j$ is the distribution associated with $g(j) \in D(E^V)$ for $j \in W$, and $\psi \in E^{V+W}$. 
\end{defn}

For $W=\{u\}, u \in \Z,$ the definition simplifies to $\alpha_{\{u\},u}(g)(\psi)=g(u)(\sigma_u(\psi))$. For $V=\{0\}$ the conditions vanish, and we get back simple multiplication, $\alpha_{W,i}(g)(\psi)=\prod_{k \in W}g(k)(\psi(k))$. We justify the general definition in the following lemma. 

\begin{lm}
Let $g \in D^W_{dBe}$ and $i \in W$. Then $\alpha_{W,i}(g) \in D(E^{V+W})$. 
\end{lm}

\begin{proof}
Let $g \in D^W_{dBe}$ and $i \in W$. $\alpha_{W,i}(g)\not \equiv 0$ because there is at least one extendable pattern with nonzero probability. We prove that it is also normalized in the following. Without loss of generality we assume that $i=u$. Then
\begin{align*}
\sum_{\varphi \in E^{V+W}} \alpha_{W,u}(g)(\varphi)=& \sum_{\varphi \in E^{V+W}} \prod_{k=-t}^{u-1} \mu_k[\{\sigma_k(\varphi)|_{\{-p\}}\}\,|\, \{\sigma_k(\varphi)|_{V_+}\}]\, \mu_u[\{\sigma_u(\varphi)|_{V}\}]\, \\
=& \sum_{\tilde \varphi \in E^{V+ W\backslash\{-t\}}} \sum_{\varphi \in \left\{\tilde \varphi |_{V+W\backslash \{-t\}}^{V+W}\right\}} \mu_{-t}[\{\sigma_{-t}(\varphi)|_{\{-p\}}\}\,|\, \{\sigma_{-t}(\tilde \varphi)|_{V_+}\}] \\
& \prod_{k=-t+1}^{u-1} \mu_k[\{\sigma_k(\tilde \varphi)|_{\{-p\}}\}\,|\, \{\sigma_k(\tilde \varphi)|_{V_+}\}]\, \mu_u[\{\sigma_u(\tilde \varphi)|_{V}\}] \\
\end{align*}
\begin{align*}
=& \sum_{\tilde \varphi \in E^{V+ W\backslash\{-t\}}} \prod_{k=-t+1}^{u-1} \mu_k[\{\sigma_k(\tilde \varphi)|_{\{-p\}}\}\,|\, \{\sigma_k(\tilde \varphi)|_{V_+}\}]\, \mu_u[\{\sigma_u(\tilde \varphi)|_{V}\}] \\
=& ... = \sum_{\varphi \in E^{u+V}}\mu_u[\{\sigma_u(\varphi)|_{V}\}]=\sum_{\varphi \in E^{V}}g(\varphi)=1.
\end{align*}
The steps indicated by ... follow by induction in $|W|$. 
\end{proof}

We find the following results for our construction. 

\begin{lm}
\label{lm:imbetaindep}
Let $i \in W$ and $g \in \text{im }(\beta_W)$. Then $\alpha_{W,i}(g)=\alpha_{W,j}(g)$ for all $i,j \in W$. 
\end{lm}

{\em Proof}
We show that $\alpha_{W,i}(g)=\alpha_{W,i+1}(g)$ for all $i \in W\backslash \{-t\}$. An index shift to the left can be proven analogously. 

Note that $\mu_i[\{\sigma_i(\psi)|_{V}\}]=g(i)(\sigma_i(\psi)|_V)$ and that for $k \in \{-t,...,i-1\}$ and $l \in \{i+1,...,u\}$
\begin{align*}
\mu_k[\{\sigma_k(\psi)|_{\{-p\}}\}\,|\, \{\sigma_k(\psi)|_{V_+}\}]&=\frac{g(k)(\sigma_k(\psi)|_V)}{\sum_{\chi \in \left\{\sigma_k(\psi)|^V_{V_+}\right\}}g(k)(\chi)}, \\
\mu_l[\{\sigma_l(\psi)|_{\{q\}}\}\,|\, \{\sigma_l(\psi)|_{V_-}\}]&=\frac{g(l)(\sigma_l(\psi)|_V)}{\sum_{\chi \in \left\{\sigma_l(\psi)|^V_{V_-}\right\}}g(l)(\chi)}. 
\end{align*}
Therefore $\alpha_{W,i}(g)(\psi)$ and $\alpha_{W,i+1}(g)(\psi)$ have the same numerator and only differ in the denominator. It is enough to show that a factor in the denominator may be shifted one step to the right: Let $i \in W\backslash \{u\}$ and $g=\beta_W(\tilde g)$ for $\tilde g \in D(E^{V+W})$. Then 
\begin{align*}
&\sum_{\chi \in \left\{\sigma_i(\psi)|^V_{V_+}\right\}}\beta_W(\tilde g)(i)(\chi)&=&\sum_{\chi \in \left\{\sigma_i(\psi)|^V_{V_+}\right\}}\, \, \sum_{\varphi \in \left\{\sigma_{-i}(\chi)|^{V+W}_{i+V}\right\}}\tilde g(\varphi) \\
=&\sum_{\varphi \in \left\{\psi|^{V+W}_{i+V_+}\right\}} \tilde g(\varphi) &=&\sum_{\varphi \in \left\{\psi|^{V+W}_{i+1+V_-}\right\}} \tilde g(\varphi) \\
=&\sum_{\chi \in \left\{\sigma_{i+1}(\psi)|^V_{V_-}\right\}}\, \, \sum_{\varphi \in \left\{\sigma_{-(i+1)}(\chi)|^{V+W}_{i+1+V}\right\}}\tilde g(\varphi) &=&\sum_{\chi \in \left\{\sigma_{i+1}(\psi)|^V_{V_-}\right\}}\beta_W(\tilde g)(i+1)(\chi) \qquad \endproof
\end{align*}

\begin{thm}
\label{thm:betaalpha}
Let $g \in \text{im}(\beta_W)$. Then $\beta_W\alpha_{W,i}(g)=g$ for all $i \in W$. 
\end{thm}

\begin{proof}
Let $i \in W$ and $g \in \text{im}(\beta_W)$. We prove $\beta_W\alpha_{W,i}(g)(j)=g(j)$ without loss of generality only for $j=u$. 

With Lm. \ref{lm:imbetaindep} and because a conditional distribution is a distribution as well, for $\psi \in E^V$
\begin{align*}
\beta_W\alpha_{W,i}(g)(u)(\psi)=&\beta_W\alpha_{W,u}(g)(u)(\psi) = \sum_{\varphi \in \left\{\sigma_{-u}(\psi)|^{V+W}_{u+V}\right\}}\alpha_{W,u} (g) (\varphi)\\
=& \sum_{\tilde \varphi \in \left\{\sigma_{-u}(\psi)|^{V+W\backslash \{-t\}}_{u+V}\right\}} \, \,\sum_{\varphi \in \left\{\tilde \varphi|^{V+W}_{V+W\backslash \{-t\}}\right\}} \mu_{-t}[\{\sigma_{-t}(\varphi)|_{\{-p\}}\}\,|\, \{\sigma_{-t}(\tilde \varphi)|_{V_+}\}] \\
& \prod_{k=-t+1}^{u-1} \mu_k[\{\sigma_k(\tilde \varphi)|_{\{-p\}}\}\,|\, \{\sigma_k(\tilde \varphi)|_{V_+}\}]\, \mu_u[\{\sigma_u(\tilde \varphi)|_{V}\}]\, \\
=& \sum_{\tilde \varphi \in \left\{\sigma_{-u}(\psi)|^{V+W\backslash \{-t\}}_{u+V}\right\}} \, \, \prod_{k=-t+1}^{u-1} \mu_k[\{\sigma_k(\tilde \varphi)|_{\{-p\}}\}\,|\, \{\sigma_k(\tilde \varphi)|_{V_+}\}]\, \mu_u[\{\sigma_u(\tilde \varphi)|_{V}\}]\, \\
=& ... =g(u)(\psi). 
\end{align*}
The last steps follow by induction in $|W|$.
\end{proof} 

Now we focus on global densities that are preserved under $\alpha_{W,i}\beta_W$ for all $i \in W$. We will see that they enable an algebraic interpretation of $\beta_W$. Our choice of $\alpha_{W,i}$ is further justified thereby. 

\begin{defn}
$g \in D(E^{V+W})$ is called $V$-factorizable, if $g=\alpha_{W,i}\beta_W(g)$ for all $i \in W$. 
\end{defn}

An example of a $\{0,1\}$-factorizable density is $\tilde g$ in Ex. \ref{ex:beta}. $g$ in the same example however has correlations over $2$ sites and is not $\{0,1\}$-factorizable: a state only has positive probability if the local states at sites $0$ and $2$ differ. The correlations are ignored by $\beta_W$. 




\begin{lm}
\label{lm:imbetaVfact}
Let $i \in W$ and $g \in \text{im }(\beta_W)$. Then $\alpha_{W,i}(g)$ is V-factorizable. 
\end{lm}

\begin{proof}
Let $g \in \text{im }(\beta_W)$ and $i,j \in W$. Then by Lm. \ref{lm:imbetaindep} and Thm. \ref{thm:betaalpha}
\begin{align*}
\alpha_{W,i}(g)=\alpha_{W,j}(g)=\alpha_{W,j}\beta_W \alpha_{W,i}(g),
\end{align*}
and hence $\alpha_{W,i}(g)$ is V-factorizable. 
\end{proof}

\begin{thm}
\label{thm:representative}
For all $g \in D(E^{V+W})$ there is a $V$-factorizable $\tilde g \in D(E^{V+W})$ such that $\beta_W(\tilde g)=\beta_W(g)$. 
\end{thm}

\begin{proof}
Let $g \in D(E^{V+W})$ and choose $\tilde g=\alpha_{W,u} \beta_W(g)$. Then $\tilde g$ is $V$-factorizable by Lm. \ref{lm:imbetaVfact}, and the definition does not depend on the site $u$. Furthermore $\beta_W(\tilde g)=\beta_W\alpha_{W,u} \beta_W(g)=\beta_W(g)$ by Thm. \ref{thm:betaalpha}. 
\end{proof}

$\beta_W$ induces equivalence classes on $D(E^{V+W})$ by collecting all global densities with the same image in one class. According to Thm. \ref{thm:representative} each equivalence class contains at least one $V$-factorizable state. Moreover, we know that there is exactly one such state, because $\beta_W$ is injective on these states by definition. It is given as the image under $\alpha_{W,i}\beta_W$ of any state in the class for any $i \in W$. Therefore it is possible to choose the $V$-factorizable states as the representatives of the equivalence classes. These representatives are preserved under $\alpha_{W,i} \beta_W$ for any $i \in W$. Ex. \ref{ex:beta} provides an example: $g$ and $\tilde g$ are in the same equivalence class, and $\tilde g=\alpha_{W,0}\beta_W(g)$ is the unique $V$-factorizable representative of the class. 

However, in general $\alpha_{W,i}(g)$ is not $V$-factorizable if $g \in D^W_{dBe}\backslash \text{im }(\beta_W)$. There is a degree of freedom in how to map a density collection to a global density on this set. We choose the arithmetic mean over all $\alpha_{W,i}$, where $i \in W$. Note that for $g \in \text{im }(\beta_W)$ the definition then coincides with any $\alpha_{W,i}$. 

\begin{defn}
$\alpha_W: D_{dBe}^W \to D(E^{V+W})$ is given by $\alpha_W(g)=\frac{1}{|W|} \sum_{i \in W}\alpha_{W,i}(g)$. 
\end{defn}

It is clear that $\alpha_W(g)$ is a density by similar reasoning as for $\alpha_{W,i}(g)$.

\subsection{General Cellular Probabilistic Automata}

\label{subsec:generalCPA}

With the de Bruijn calculus at hand we can now generalize the definition of CPA to general $V$. 

\begin{defn} As before let $I=\{1,...,m\}$, $U=\{-r,...,s\}$, $V=\{-p,...,q\}$ for $p,q,r,s,m \in \N_0$ with $m \geq 1, 1+p+q+r+s\leq m$ and $K=\{1,...,r\}\cup \{m-s+1,...,m\}$. We now set $i_l=1+p+r$, $i_r=m-q-s$ and $\tilde I = \{i_l,...,i_r\}$. 
\begin{compactitem}
\item [i)] We set $\hat \alpha: D_{dBe}^{\tilde I} \to D_{BC}(E^I), \,g \mapsto \hat \alpha(g)$ with
\begin{align*}
\hat \alpha(g)(\psi) &= \left\{\begin{array}{ll}
\alpha_{\tilde I}(g)(\psi|_{\{1+r,...,m-s\}}) & \mbox{ if  } \psi |_K=\rho\\
0 & \mbox{ else } \\
\end{array}\right. 
\end{align*}
\item [ii)] and $\hat \beta: D_{BC}(E^I) \to D_{dBe}^{\tilde I}, \, g \mapsto \hat \beta (g)$ with 
$$\hat \beta (g)(i)(\psi)=\sum_{\overset{\chi \in E^I \text { s. t. }} {\chi |_{i+V}=\sigma_{-i}(\psi)}}g(\chi).$$
\end{compactitem}
\end{defn}

\begin{defn}
\label{defn:CPA}
A cellular probabilistic automaton (CPA) is a tuple $(I, U, V, E, f_0)$, where for $m,p,q,r,s \in \N_0$ with $m\geq 1$ and $1+p+q+r+s\leq m$
\begin{compactitem}
\item [i)] $I=\{1,...,m\}$ is a finite grid, 
\item [ii)] $U=\{-r,...,s\}$ is the neighborhood, 
\item [iii)] $V=\{-p,...,q\}$ gives rise to de Bruijn patterns, 
\item [iv)] $E$ is a finite set of local states
\item [v)] and $f_0 : E^{U+V} \rightarrow D(E^{V})$ is the local function. 
\end{compactitem}

With the boundary conditions $\rho \in E^K$ for $K=\{1,...,r\}\cup \{m-s+1,...,m\}$ and the definitions $i_l=1+p+r$, $i_r=m-q-s$, $\tilde I = \{i_l,...,i_r\}$, $\tilde U(i)=\{-\underline u(i),...,\overline u(i)\}$ and $\underline u,\overline u:\tilde I \to U$ with 
\begin{align*}
\underline u(i) &= \left\{\begin{array}{ll}
i-i_l & \mbox{ if  } i \in \{i_l,...,i_l+r-1\}\\
r & \mbox{ else } \\
\end{array}\right., \\
\overline u(i) &= \left\{\begin{array}{ll}
i_r-i & \mbox{ if  } i \in \{i_r-s+1,...,i_r\}\\
s & \mbox{ else } \\
\end{array}\right., 
\end{align*}
the global function is given by
\begin{equation*}
f: D_{dBe}^{\tilde I} \rightarrow D_{dBe}^{\tilde I}, \quad g \mapsto f(g), 
\end{equation*}
$$f(g)(i)(\psi) = \sum_{\overset{\varphi \in E^{U+V} \text{ s.t. } \varphi(k-i)=\rho(k)}{\text{for }k \in K\cap i+U+V}}\alpha_{\tilde U(i)}(\sigma_i(g)|_{\tilde U(i)})(\varphi|_{V+\tilde U(i)})\cdot f_0(\varphi)(\psi).$$
The trajectory starting with  $g^0 \in (D(E^{V}))^{\tilde I}$ is given by the sequence $(g^n)_{n \in \N}$, where $g^n=f(g^{n-1})$ for $n \in \N ^+$.
\end{defn}

See Fig. \ref{fig:CPAdef} for a sketch of how the CPA works on general patterns. Remark that $f_0$ is not arbitrary but connected to a dynamical system with the locality property. By investigating this relation we can assure that the global function is well-defined. 

\begin{figure}
\centering
\includegraphics[width=8cm]{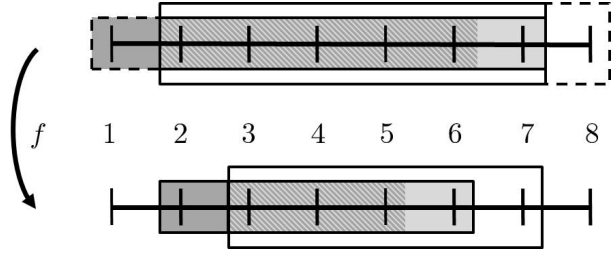}
\caption{An example of a CPA with $I=\{1,...,8\}$, $U=\{-1,0,1\}$ and $V=\{-2,...,2\}$. The global function considers patterns located at $\tilde I=\{4,5\}$ and sketched in the image (lower part) by rectangles of small and large height, respectively. The corresponding preimage patterns (upper part) are larger due to the neighborhood, and their probability of occurence is influenced by boundary conditions (dashed parts). From an implementational point of view, $V$ may be constructed from $\tilde V=\{-2,...,1\}$ and $W=\{0,1\}$, see Section \ref{subsec:Implementation}: Focus on a pattern at site $4$. The transition probability from a preimage pattern is calculated from the information about two subtransitions between light and dark grey subpatterns. }
\label{fig:CPAdef}
\end{figure}

\begin{lm}
$f(D_{dBe}^{\tilde I}) \subseteq D_{dBe}^{\tilde I}$. 
\end{lm}

\begin{proof}
Because $D_{dBe}^{\tilde I}=(D(E^V))^{\tilde I}$ for $|V|=1$, the statement is trivial in this case. So we focus on $|V|>1$. We show that without loss of generality any pattern in the support of any site in the image can be extended to the right by a pattern in the support of the neighboring site. 

Let $g \in D_{dBe}^{\tilde I}, i \in \{i_l,...,i_r-1\}$ and $\psi \in \text{supp }f(g)(i)$. By the construction of $f$ and $\alpha_{\tilde U(i)}$ we know that there is $\varphi \in E^{U+V}$ such that $\sigma_{j}(\varphi)|_{V} \in \text{supp}(g)(i+j)$ for all $j \in U$ and $f_0(\varphi)(\psi)>0$. Because of the extension property of the preimage $g$ we can find $\tilde \varphi \in E^{U+V}$ such that $\sigma_j(\tilde \varphi)|_{V}=\sigma_{j+1}(\varphi)|_V \in \text{supp }(g)(i+1+j)$ for all $j \in \{-r,...,s-1\}$ and $\sigma_s(\tilde \varphi)|_{V} \in \text{supp }(g)(i+1+s)$. 

This enables us to find $\tilde \psi \in \text{supp }f(g)(i+1)$ such that $\tilde \psi|_{V_-}=\sigma_1(\psi|_{V_+})$, as we will show in the following. Hence the pattern $\psi$ may be extended to the right by $\tilde \psi$, and the proof is complete. 

Since $f_0(\varphi)(\psi)>0$ and the partition is uniform, there is an $\epsilon$-ball $B_{\epsilon}$ with respect to the 2-norm in $\R^{mn}$, $\epsilon >0,$ such that 
$$B_{\epsilon}\subseteq \Omega_{\sigma_{-i}(\varphi)} \cap \Phi^{-\tau}(\Omega_{\sigma_{-i}(\psi)}).$$ 
Because the set is just restricted on sites $i+U+V$ due to the locality property, we may independently restrict at site $i+1+q+s$ and can still find $\epsilon '>0$ with 
\begin{align*}
B_{\epsilon '}&\subseteq \Omega_{\sigma_{-i}(\varphi)} \cap \Omega_{\sigma_{-(i+1)}(\tilde \varphi(q+s))} \cap \Phi^{-\tau}(\Omega_{\sigma_{-i}(\psi)})\\
&\subseteq \Omega_{\sigma_{-i}(\varphi|_{V_++U})} \cap \Omega_{\sigma_{-(i+1)}(\tilde \varphi|_{q+s})} \cap \Phi^{-\tau}(\Omega_{\sigma_{-i}(\psi|_{V_+})})\\ 
&= \Omega_{\sigma_{-(i+1)}(\tilde \varphi)} \cap \Phi^{-\tau}(\Omega_{\sigma_{-i}(\psi|_{V_+})}).
\end{align*}
In the second line we have again used the locality property, and the equality sign holds due to $\sigma_{-i}(\varphi|_{V_++U})=\sigma_{-(i+1)}(\tilde \varphi|_{V_-+U})$. We now define $\tilde \psi \in E^V$ by $\tilde \psi|_{V_-}=\sigma_1(\psi|_{V_+})$ and choose $\tilde \psi (q)$ such that there is $\epsilon ''>0$ with 
$$B_{\epsilon ''}\subseteq B_{\epsilon '} \cap \Phi^{-\tau}(\Omega_{\sigma_{-(i+1)}(\psi(q))}).$$ 
Therefore  
$$B_{\epsilon ''}\subseteq \Omega_{\sigma_{-(i+1)}(\tilde \varphi)} \cap \Phi^{-\tau}(\Omega_{\sigma_{-(i+1)}(\tilde \psi)}),$$
$f_0(\tilde \varphi)(\tilde \psi)>0$, and $\tilde \psi \in \text{supp }f(g)(i+1)$. 
\end{proof}

We denote the case of $i_l=i_r$ with $V_{\max}$ and find that $\hat \alpha \hat \beta (g)= g$ for all $g \in {D_{BC}(E^I)}$ and $\hat \beta \hat \alpha (g)=g$ for all $g \in (D(E ^{V}))^{\tilde I}$. Furthermore $\tilde U(\tilde I)=\{0\}$ in this case, and it can be calculated for $g \in D_{dBe}^{\tilde I}$ and $\varphi \in E^{U+V}$ that 
$$\alpha_{\{0\}}(\sigma_{i_l}(g))(\varphi|_V)=g(i_l)(\varphi|_V).$$
For $\psi \in E^V$ then
$$f(g)(i_l)(\psi)=\sum g(i_l)(\varphi|_V) \cdot f_0(\varphi)(\psi),$$
where the sum is taken over all $\varphi \in E^{U+V}$ such that the boundary conditions are fulfilled. For $V_{\max}$ the evolution of the global density is calculated directly, and locality is completely omitted. 



\section{Consistency}

\label{sec:Consistency}

The goal is to prove in Section \ref{subsec:consistencythm} that time evolution of probability densities under CPA can approximate time evolution under FPO arbitrarily close. We prepare this consistency result by investigating first the two features state space discretization and locality in more depth in Sections \ref{subsec:statespacedis} and \ref{subsec:Locality}, respectively.

\subsection{State Space Discretization}
\label{subsec:statespacedis}

In this section we just compare the discretized FPO to the real FPO without considering locality. There are many ways to study distances between probability measures \cite{probmetricsreview}. In our density based formulation we use the $\mathcal{L}^1$-norm for probability densities which leads to the notion of strong convergence in the literature \cite{Lasota}. It is well-known that in this norm the discretized operator converges pointwise to the FPO for increasing state space resolution in the case of Lipschitz continuous input densities \cite{Koltai}. But because the image under $RP_{\Phi^\tau}$ is in general not continuous, for iteration we need to generalize the result to $\mathcal{L}^1$-functions. We start with some necessary tools. 

A subset $M\subseteq X$ of a metric space $(X,d)$ is called totally bounded, if for every $\epsilon >0$ there exist $n \in \N$ and $x_1,...,x_n \in M$ such that 
$$M\subseteq\bigcup_{i=1}^n\{x \in X \,|\, d(x,x_i)<\epsilon\}.$$
Totally bounded subsets of $\mathcal{L}^p(\R^n)$ can be alternatively characterized in a functional analytical sense by the theorem of Kolmogorov-Riesz \cite{Hanche-OlsenHolden, Wloka}. In the following $|.|$ denotes the 2-norm in $\R^n$. 

\begin{thm}
\label{thm:Kolmogorov-Riesz}
Kolmogorov-Riesz

A set $M\subset \mathcal{L}^p(\R^n)$, $1\leq p<\infty$ is totally bounded if and only if the following criteria are fulfilled: 
\begin{compactitem}
\item [i)] $M$ is bounded in $\mathcal{L}^p(\R^n)$, i.e. $\sup_{g \in M} \|g\|_p<\infty$.
\item [ii)] For every $\epsilon>0$ there is some $R$ so that for every $g \in M$ 
$$\int_{|v|>R} |g(v)|^pdv < \epsilon^p.$$
\item [iii)] For every $\epsilon>0$ there is $\delta>0$ so that for every $g \in M$ and $w \in \R^n$ with $|w|<\delta$ 
$$\int_{\R^n}|g(v+w)-g(v)|^pdv<\epsilon^p.$$
\end{compactitem}
\end{thm}

\begin{thm}
\label{thm:L1conv}
Let $g \in D(\R^{mn})$ with $\text{supp}(g)\subseteq \Omega ^m$ and $T: \Omega \to E$ a uniform partition with resolution $\Delta \Omega$. Then $R$ converges pointwise to the identity with respect to the $\mathcal{L}^1$-norm, 
$$\|R(g)-g\|_1\to 0 \quad (\Delta \Omega \to 0).$$
\end{thm}

\begin{proof}
\begin{align*}
\|R(g)-g\|_1&=\int_{v \in \Omega ^m} \left|\sum_{\varphi \in E^I} \chi_{\Omega_{\varphi}}(v)\frac{1}{\Delta \Omega^{mn}}\int_{w \in \Omega_{\varphi}}g(w)dw-\sum_{\varphi \in E^I} \chi_{\Omega_{\varphi}}(v)g(v)\right|dv\\
&\leq \sum_{\varphi \in E^I} \frac{1}{\Delta \Omega^{mn}}\int_{v \in \Omega_{\varphi}}\int_{w \in \Omega_{\varphi}}\left|g(w)-g(v)\right|dwdv\\
&= \sum_{\varphi \in E^I} \frac{1}{\Delta \Omega^{mn}}\int_{v \in \Omega_{\varphi}}\int_{u \in \Omega_{\varphi}-v}\left|g(v+u)-g(v)\right|dudv\\
\end{align*}

The last step involves a change of variables from $w$ to $u:=w-v$, and $\Omega_{\varphi}-v:=\{u-v \,|\, u \in \Omega_{\varphi}\}$. As 
$$\Omega_{\varphi}-v\subseteq [-\Delta \Omega, \Delta \Omega]^{mn}$$
for $v \in \Omega_{\varphi}$ with $\varphi \in E^I$, we calculate with Fubini's theorem 
\begin{align*}
\|R(g)-g\|_1&\leq \sum_{\varphi \in E^I} \frac{1}{\Delta \Omega^{mn}}\int_{v \in \Omega_{\varphi}}\int_{u \in [-\Delta \Omega, \Delta \Omega]^{mn}}\left|g(v+u)-g(v)\right|dudv\\
&\leq \frac{1}{\Delta \Omega^{mn}}\int_{u \in [-\Delta \Omega, \Delta \Omega]^{mn}}\int_{v \in \R^{mn}}\left|g(v+u)-g(v)\right|dvdu.\\
\end{align*}
Let $\epsilon >0$. Because the set $\{g\}\subset \mathcal{L}^1(\R^n)$ is totally bounded, the theorem of Kolmogorov-Riesz, Thm. \ref{thm:Kolmogorov-Riesz}, guarantees that there is $\delta>0$ such that 
$$\int_{v \in \R^{mn}}\left|g(v+u)-g(v)\right|dv<\frac{\epsilon}{2^{mn}}$$
for $|u|<\delta$. If we choose $\Delta \Omega$ such that $\Delta \Omega<\frac{\delta}{\sqrt{mn}}$ we ensure that $$|u|=\sqrt{\sum_{i=1}^{mn}|u_i|^2}\leq\sqrt{ \sum_{i=1}^{mn}\Delta \Omega^2}=\sqrt{mn}\Delta \Omega<\delta$$ 
for all $u \in [-\Delta \Omega, \Delta \Omega]^{mn}$ and hence
\begin{align*}
\|R(g)-g\|_1&< \frac{1}{\Delta \Omega^{mn}}\int_{u \in [-\Delta \Omega, \Delta \Omega]^{mn}}\frac{\epsilon}{2^{mn}}du=\epsilon.
\end{align*}
Therefore $\|R(g)-g\|_1 \to 0$ for $\Delta \Omega \to 0$. \qquad
\end{proof}

\subsection{Locality}
\label{subsec:Locality}

In this section we investigate in more depth the role of locality in approximating the discretized FPO $P_B$ by a CPA. It turns out that the CPA covers the dynamics of the underlying $P_B$ if only the support is considered. However, we cannot obtain the precise behavior of the global density with CPA in general, as will be shown with two examples. In the end we will see that according errors vanish for maximal pattern size. 

Let us start with the cover property. 

\begin{lm}
For all $n \in \N$, for all $g \in D_{BC}(E^I)$ and all $i \in \tilde I$ it holds that 
$$\text{supp }(\hat \beta P_B^n (g)(i)) \subseteq \text{supp }(f^n \hat \beta (g)(i)).$$
\end{lm}

\begin{proof}
Let $n \in \N$, $g \in D_{BC}(E^I)$, $i \in \tilde I$ and $\chi \in \text{supp }(\hat \beta P_B^n(g)(i)) \in E^V$. Then there is $\psi \in E^I$ such that $P_B^n(g)(\psi)>0$ and $\sigma_i(\psi)|_V=\chi$. Let $\varphi_n=\psi$. Per induction it can be shown that we can find $\varphi_0,...,\varphi_{n-1} \in E^I$ such that $P_{B,\varphi_{k-1},\varphi_k}=\frac{\lambda(\Omega_{\varphi_{k-1}}\cap\Phi^{-\tau}(\Omega_{\varphi_n}))}{\lambda(\Omega_{\varphi_{n-1}})}>0$
and $P_B^{k-1}(g)(\varphi_{k-1})>0$ for $k \in \{1,...,n\}$. Because for all $j \in \tilde I$
\begin{align*}
\Omega_{\varphi_{k-1}}\cap \Phi^{-\tau}(\Omega_{\varphi_k}) &\subseteq \Omega_{\varphi_{k-1}|_{j+U+V}}\cap \Phi^{-\tau}(\Omega_{\varphi_n}|_{k+U+V}), \\
\Omega_{\varphi_{k-1}} & \subseteq \Omega_{\varphi_{k-1}}|_{j+U+V}, 
\end{align*}
we conclude that $f_0(\sigma_j(\varphi_{k-1})|_{U+V})(\sigma_j(\varphi_k)|_V)>0$ for all $j \in \tilde I$. Furthermore $\hat \beta (g)(j)(\sigma_j(\varphi_{0}))>0$ for all $j \in \tilde I$, and therefore $f\hat \beta (g)(j)(\sigma_j(\varphi_1)|_V)>0$ for all $j \in \tilde I$. This induces $f^2\hat \beta(g)(j)(\sigma_j(\varphi_2)|_V)>0$ for all $j \in \tilde I$ and so on, and therefore $f^n\hat \beta (g)(j)(\sigma_j(\varphi_n)|_V)>0$ for all $j \in \tilde I$. Recalling that $\sigma_i(\varphi_n)|_V=\chi$, we conclude that $\chi \in \text{supp }(f^n\hat \beta (g)(i))$. 
\end{proof}

However, we cannot recover the precise global behavior of the discretized FPO from a CPA. The errors that can occur are twofold, and we will provide examples for both types here. On the one hand it may happen that correlations over $|V|$ sites are not preserved because we work on patterns of size $V$. On the other hand we will see that even for $U \subseteq V$ in general there are locally allowed transitions of a global state that are not allowed in a global consideration with $P_B$. This is remarkable, since such behavior was ruled out for the underlying dynamical system by the locality property. But because a CPA only computes locally, that may also lead to errors. While the first error type is a true locality effect, the second arises from the interplay of locality and state space discretization. 

\begin{ex}
This example shows that in general correlations over $|V|$ sites are not preserved. We compare one CPA time step to one time step with the discretized FPO. Consider $I=\{1,2,3\}$, and the dynamical system that is given by the identity on $\Omega ^m=[0,1]^2$, i.e. $U=\{0\}$ and $K=\emptyset$. We choose the partition $\Omega_0=[0,0.5)$ and $\Omega_1=[0.5,1]$, i.e. $E=\{0,1\}$, and look at the CPA with $V=\{0,1\}$ and $\tilde I=\{1,2\}$. We find $f_0(\varphi)(\psi)=\delta_{\varphi, \psi}$  for all $\varphi, \psi \in E^{\{0\}}$. 

Consider $g \in D_{BC}(E^I)=D(E^{V+W})$ with $W=\tilde I$ from Ex. \ref{ex:beta}. $\hat \beta (g)=\beta_{\tilde I}(g)$, and also $f\hat \beta (g)=\beta_{\tilde I}(g)$. So $\hat \alpha \hat \beta (g)=\alpha_{\tilde I} \beta_{\tilde I}(g)=\tilde g$. However, $P_B(g)=g$, and so $\hat \alpha f \hat \beta (g) \neq P_B(g)$. 
\end{ex}


\begin{figure}
\centering
\includegraphics[width=10cm]{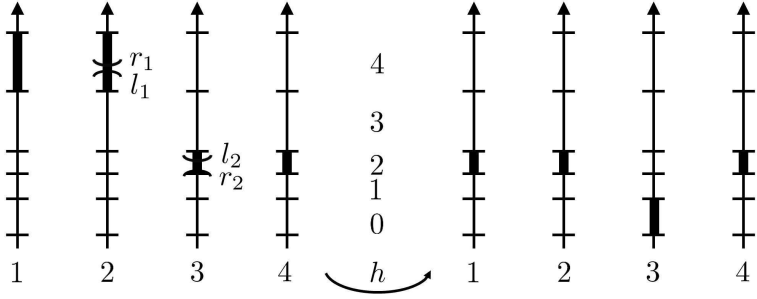}
\caption{Illustration of a transition from $\chi=(4,4,2,2)$ to $\psi=(2,2,0,2)$ in Ex. \ref{ex:locerror}. The left side shows the preimage, the right the image state. The horizontal numbers correspond to the respective sites, while the vertical numbers display $E=\{0,...,4\}$. The states $\chi$ and $\psi$ are marked by black rectangles at the corresponding sites. }
\label{fig:localityex}
\end{figure}

\begin{ex}
\label{ex:locerror}
This example shows that transitions at different sites are not independent in general. By comparing one CPA time step to one time step with the discretized FPO we see that a specific local transition at one site cannot take place if another specific local transition happens at a neighboring site, although both transitions are allowed locally. Consider $I=\{1,...,4\}$, and the system on dynamically invariant state space $\Omega ^m=[0,1]^4$ given for all $n \in \N$ by $v_{i}^{n+1}=h(v_i^n,v_{i+1}^n)=\frac{v_i^n+v_{i+1}^n}{3.75}$ for $i \in \{1,...,m-1\}$. We have $U=\{0,1\}$ and define $5$ intervals 
$$\Omega_j=[w_j,w_{j+1}), \quad \text{ for } j\in \{0,...,3\}, \quad \Omega_{4}=[w_{4},w_5]$$
with 
$$w_0 =0, \quad w_1=0.183, \quad w_2=0.31, \quad w_3=0.4, \quad w_4=0.7, \quad w_5=1,$$
name them by their index and obtain a partition of $\Omega$ with $E=\{0,...,4\}$, see Fig. \ref{fig:localityex}. The induced flow is denoted by $\Phi^1$ for one time step. We consider the CPA with $V=U$ for deterministic input $g \in D_{BC}(E^I)$ given by $g(\varphi)=\delta_{\chi,\varphi}$, where $\chi=(4,4,2,2)\in E^{I}$. We focus on the image state $\psi =(2,2,0,2)\in E^{I}$ and determine 
$$l_1=0.8, \quad l_2=0.3625, \quad r_1=0.83875, \quad r_2=0.32375$$
as the solution of the equations 
$$h(w_4,l_1)=w_3, \quad h(l_1,l_2)=w_2, \quad h(r_1,r_2)=w_2, \quad h(r_2,w_2)=w_1.$$
It is possible to show that 
\begin{align*}
\{\Omega _{\chi|_{1+U+V}}\cap \Phi^{-1}(\Omega _{\psi|_{1+V}})\} &\subseteq \{v \in \Omega \,|\, w_4\leq v_2 \leq l_1, l_2\leq v_3 \leq w_3\}, \\
\{\Omega _{\sigma_1(\chi|_{2+U+V})}\cap \Phi^{-1}(\Omega _{\sigma_1(\psi|_{2+V})})\} &\subseteq \{v \in \Omega \,|\, r_1\leq  v_2 \leq 1, w_2 \leq v_3 \leq r_2\},
\end{align*}

$f_0(\sigma_1(\chi|_{1+U+V}))(\sigma_1(\psi|_{1+V})))>0$ and $f_0(\sigma_2(\chi|_{2+U+V}))(\sigma_2(\psi|_{2+V}))>0$. Hence $\hat \alpha f \hat \beta (g)(\psi)>0$, but 
\begin{align*}
P_B(g)(\psi)&=\sum_{\varphi \in E^I}g(\varphi)P_{B,\varphi,\psi}=\frac{\lambda(\Omega _{\chi}\cap \Phi^{-1}(\Omega _{\psi}))}{\lambda(\Omega _{\chi})}\\
\leq &\frac{\lambda((\Omega _{\chi|_{1+U+V}}\cap \Phi^{-1}(\Omega _{\psi|_{1+V}}))\cap(\Omega _{\chi|_{2+U+V}}\cap \Phi^{-1}(\Omega _{\psi|_{2+V}})))}{\lambda(\Omega _{\chi})}\\
=&\frac{\lambda(\emptyset)}{\lambda(\Omega _{\chi})}=0,
\end{align*}
and so $\hat \alpha f \hat \beta (g)\neq P_B (g)$. 
\end{ex}

Both examples are scalable in the sense that we can find analogous partitions of $[0,c], c \in (0,1),$ with the above properties by dividing all phase space coordinates by $c$ and complete the partition in $[c,1]$ arbitrarily. So for decreasing size of the coding domains we can still find a partition of $[0,1]$ with the above effects: locality errors are independent from resolution errors. 

Furthermore the examples suggest to choose large $V$ for good approximations. Accordingly it can be proven that for maximal $V$ the CPA exactly corresponds to the discretized FPO. 

\begin{prop}
\label{prop:VmaxeqFPO}
For $V=V_{\max}$ we find that $\hat \alpha f \hat \beta=P_B$. 
\end{prop}

\subsection{Consistency Theorem}
\label{subsec:consistencythm}

The preceding results can be composed to a consistency result for uncertainty propagation with CPA: Up to technical postprocessing time evolution of probability densities with CPA can approximate time evolution with the FPO arbitrarily close. 

\begin{cor}
Let $T: \Omega \to E$ be a uniform partition with resolution $\Delta \Omega$ and $g \in \kappa_B^{-1}(D_{BC}(E^I))$. Then
$$\lim_{\Delta \Omega \to 0} \lim_{V \to V_{max}}\|\kappa_B ^{-1} \hat \alpha f \hat \beta \kappa_B(g) - P_{\Phi^\tau}(g)\|_{1}=0.$$
\end{cor}

{\em Proof}
Because there is only a finite number of possible $V$, the limit $V \to V_{\max}$ is taken in the discrete topology. So the limit point is the evaluation with $V_{\max}$, and Prop. \ref{prop:VmaxeqFPO} can be used. In a second step Thm. \ref{thm:L1conv} can be applied to $P_{\Phi^\tau }(g)$: 
\begin{align*}
&\lim_{\Delta \Omega \to 0} \lim_{V \to V_{max}}\|\kappa_B ^{-1} \hat \alpha f \hat \beta \kappa_B(g) - P_{\Phi^\tau}(g)\|_{1}\\
=& \lim_{\Delta \Omega \to 0}\|\lim_{V \to V_{\max}}\kappa_B ^{-1} \hat \alpha f \hat \beta \kappa_B(g) - P_{\Phi^\tau}(g)\|_{1}\\
=& \lim_{\Delta \Omega \to 0}\|RP_{\Phi^\tau}(g) - P_{\Phi^\tau}(g)\|_{1}\\
=& 0. \qquad\endproof
\end{align*}

\section{Example}

\label{sec:Example}

In this section we comment on how to implement an algorithm to evolve uncertainties with CPA. The algorithm is tested at the problem of arsenate transportation and asorption in drinking water pipes, and the results are compared to a Monte Carlo calculation. Although the theory has been developed for uncertainties in initial conditions, in this applicational part we extend the concept slightly such that CPA can cope with certain stochastic boundary conditions. This is important in contaminant transport modeling \cite{randomsources}. With this first generalization we want to provide evidence that with CPA the treatment of more general stochastic spatio-temporal processes seems feasible. 

\subsection{Stochastic Boundary Conditions} 

We deal with stationary temporal white noise boundary conditions 
\begin{align*}
g_l \in D(E^{K_l}) &\text{ for } K_l=\{1,...,r\},\\
g_r \in D(E^{K_r}) &\text{ for } K_r= \{m-s+1,...,m\}
\end{align*}
instead of deterministic $\rho \in E^K$. With stationary we mean that the densities do not change in time, and the term temporal white noise indicates that there are no correlations in the boundary random variable's realizations at different times. 

For this purpose the global function in Def. \ref{defn:CPA} is extendend to
\begin{equation*}
f: D_{dBe}^{\tilde I} \rightarrow D_{dBe}^{\tilde I}, \quad g \mapsto f(g), 
\end{equation*}
\begin{align*}
\underline g(i)(\varphi)&=\sum_{\overset{\chi \in E^{K_l} \text{ s.t. }\sigma_{i-i_l}(\chi)|_{\{1,...,r+i_l-i\}}=}{\sigma_{-i_l}(\varphi)|_{\{1,...,r+i_l-i\}}}}g_l(\chi),\\
\overline g(i)(\varphi)&=\sum_{\overset{\chi \in E^{K_r} \text{ s.t. }\sigma_{i-i_r}(\chi)|_{\{m-s+1-i+i_r,...,m\}}=}{\sigma_{-i_r}(\varphi)|_{\{m-s+1-i+i_r,...,m\}}}}g_r(\chi),\\
f(g)(i)(\psi) &= \sum_{\varphi \in E^{U+V}}\underline g(i)(\varphi)\cdot \alpha_{\tilde U(i)}(\sigma_i(g)|_{\tilde U(i)})(\varphi|_{V+\tilde U(i)})\cdot \overline g(i)(\varphi) \cdot f_0(\varphi)(\psi).\\
\end{align*}

Furthermore the relations between global and de Bruijn densities then have to be generalized to 
 $\hat \alpha: D(E^{K_l})\times D_{dBe}^{\tilde I} \times D(E^{K_r}) \to D(E^I), \,(g_l\times g \times g_r) \mapsto \hat \alpha((g_l\times g \times g_r))$ with
\begin{align*}
\hat \alpha((g_l\times g \times g_r))(\psi) &= g_l(\psi|_{K_l})\alpha_{\tilde I}(g)(\psi|_{\{1+r,...,m-s\}})g_r(\psi|_{K_r}) 
\end{align*}
and 
$\hat \beta: D(E^I) \to D(E^{K_l})\times D_{dBe}^{\tilde I} \times D(E^{K_r}), \, g \mapsto \hat \beta (g)=(g_l,g_{\tilde I},g_r)$ with 
\begin{align*}
g_l(\psi)&=\sum_{\chi \in E^I \text{ s. t. } \chi|_{K_r}=\psi}g(\chi), \\
g_{\tilde I}(i)(\psi)&=\sum_{\overset{\chi \in E^I \text { s. t. }} {\chi |_{i+V}=\sigma_{-i}(\psi)}}g(\chi),\\
g_r(\psi)&=\sum_{\chi \in E^I \text{ s. t. } \chi|_{K_l}=\psi}g(\chi).\\
\end{align*}

CPA with stochastic boundary conditions may be used to approximate spatio-temporal processes with deterministic dynamics, in which the initial and boundary conditions are stochastic. We remark that it is straight-forward to use time-dependent stochastic boundary conditions instead of stationary ones. 

\subsection{Implementation}

\label{subsec:Implementation}

From an implementational point of view two steps of uncertainty propagation with CPA have to be distinguished. Step one is the tranlation of the completely continuous system into a CPA. This is independent of initial or boundary conditions and can be achieved in a preprocessing procedure. Step two consists of the CPA evolution with given initial and boundary values. It turns out that step one is numerically more expensive than step two. For industrial applications like simulation based system monitoring the CPA method points towards real-time uncertainty quantification, because the slow step one only has to be performed once before the actual simulation. We furthermore note that by construction the simulation is parallelizable in space. 

Step one basically consists of the approximation of local transition probabilities. We propose a local version of the standard Monte Carlo quadrature approach \cite{Hunt} in set oriented numerics for this purpose. We remark that also advanced adaptive methods have been suggested, see e. g.~\cite{Guder}.

\begin{compactitem}
\item [i)] For $\varphi \in E^{U+V}$ choose $W_\varphi$ test vectors $w_i=(w_{i,-r-p},...,w_{i,s+q}) \in (\R^n)^{U+V}$, where $\{w_{i,j}\}_{i \leq W_\varphi}$ is randomly distributed over coding domain $\Omega_{\varphi(j)} \subseteq \Omega$, respectively. 
\item [ii)] Compute for all $i \leq W_\varphi$ the image points
$$\tilde w_i=(h(\tau,w_{i,-r-p},...,w_{i,s-p}),h(\tau,w_{i,-r-p+1},...,w_{i,s-p+1}),...,h(\tau,w_{i,-r+q},...,w_{i,s+q})).$$ 
\item [iii)] Determine $\psi_1,...,\psi_L \in E^{V}$ such that there is $l \leq L$ and $\tilde w_i$ with $T((\tilde w_i)_j)=(\psi_l)_j$ for all $j \in V$. Let the number of image points in the specific coding domain be denoted by $W_{\psi_l}$, i.e. $\sum_{l=1}^L W_{\psi_l}=W_\varphi$. The local transition function is then approximated by
$$f_0(\varphi) (\psi)= \left\{\begin{array}{ll}
W_{\psi}/W_\varphi & \mbox{ for all } \psi \in \{\psi_1,...,\psi_L\}\\
0 & \mbox{ else } \\
\end{array}\right..$$
\end{compactitem}

With increasing number of test points the approximation is expected to get better. However, the number of evaluations grows exponentially in $|V|$. So we suggest to use de Bruijn calculus to determine transition probabilities for large $V$ from transition probabilities for smaller $\tilde V$, see Fig. \ref{fig:CPAdef}. For given $\tilde f_0:E^{U+\tilde V} \to D(E^{\tilde V})$ and $W$ given by $V=\tilde V+W$ 
$$f_0:E^{V+U} \to D(E^{V}), \quad \varphi \mapsto f_0(\varphi),$$
where 
$$f_0(\varphi)(\psi)=\alpha_W(\hat g)(\psi),$$
$$\hat g \in (D(E^{\tilde V}))^W, \quad \hat g (i)=\tilde f_0(\varphi|_{i+\tilde V+U}).$$
It can be shown with an example similar to Ex. \ref{ex:locerror} that this is again just an approximation of the directly calculated $f_0$. 

Regarding step two, the simulation with CPA, we remark that it is important to only follow and store states with probability larger than a specified threshold whenever possible. Otherwise already for reasonably large $E$ or $V$ the calculations are not feasible. An example is the de Bruijn density $D_{dBe}^{\tilde I}$, where $|E|^{|V|}$ numbers would have to be handled at every site in $\tilde I$. The set of states with positive probability is much smaller, although it typically first grows and then shrinks again in the transient phase of dynamics. Note that our de Bruijn choice of $\alpha_W$ enables such sparse calculations, whereas the whole space is needed to solve for example a linear nonnegative least squares problem.

\subsection{Arsenate Fate in Water Pipe}

Consider the advection and adsorption of arsenate in drinking water pipes, a topic that has attracted a lot of attention in the water supply community lately \cite{EPANETMSXmanual, Klostermannexp}. We describe a water tank on a hill and a pipe to a consumer in a valley. Report locations to observe the arsenate concentrations are installed in a distance of $\Delta x$, see Fig. \ref{fig:setup}. The physics is described by the Langmuir adsorption model \cite{Koopal}  
\begin{equation*}
\begin{split}
\partial_tD+v\partial_xD&=-\frac{1}{r_h}\frac{1}{\frac{1}{k_1}+\frac{1}{k_f}(S_{max}-A)}(D(S_{max}-A)-K_{eq}A), \\
\partial_tA&=\frac{1}{\frac{1}{k_1}+\frac{1}{k_f}(S_{max}-A)}(D(S_{max}-A)-K_{eq}A)
\end{split}
\end{equation*}
where $D$ is the concentration of dissolved arsenate in $\frac{mg}{l}$ and $A$ the concentration of arsenate adsorbed at the pipe wall in $\frac{mg}{m^2}$. We adopt realistic parameter values from \cite{Klostermannexp, Pierce} 
\begin{align*}
v&=10 \frac{m}{min}, & r_h&=50 \frac{l}{m^2},\\
k_1&=0.2 \frac{l}{mg \, min}, & S_{max}&=100 \frac{mg}{m^2},\\
K_{eq}&= 0.0537 \frac{mg}{l}, & k_f&=2.4 \frac{l}{m^2 \, min},\\
\end{align*}
and consider the system on the approximately positively invariant $\Omega$ given by $D \in [0, 1]$ and $A\in [0, S_{max}]$. The backward difference with $U=\{-1,0\}$, $\Delta x=100m$ and $\Delta t=\Delta x/v=10min$ is used. 

\begin{figure}
\centering
\subfloat[\label{fig:setup}]{\includegraphics[width=15cm]{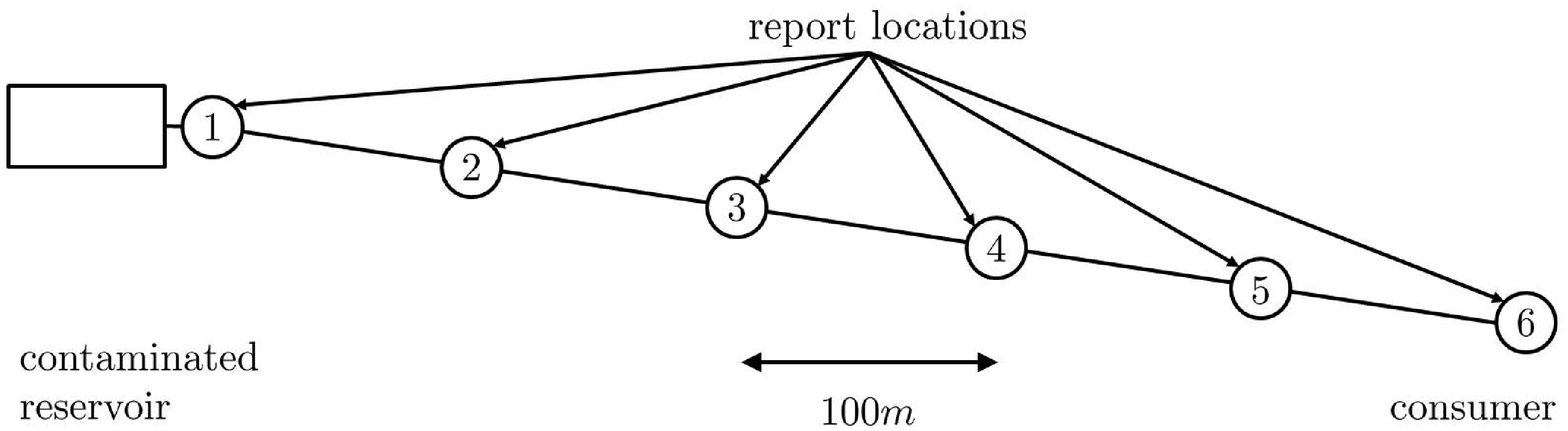}}

\hspace{1.25cm}
\subfloat[\label{fig:phasespace_C5P5}]{\includegraphics[width=5.5cm]{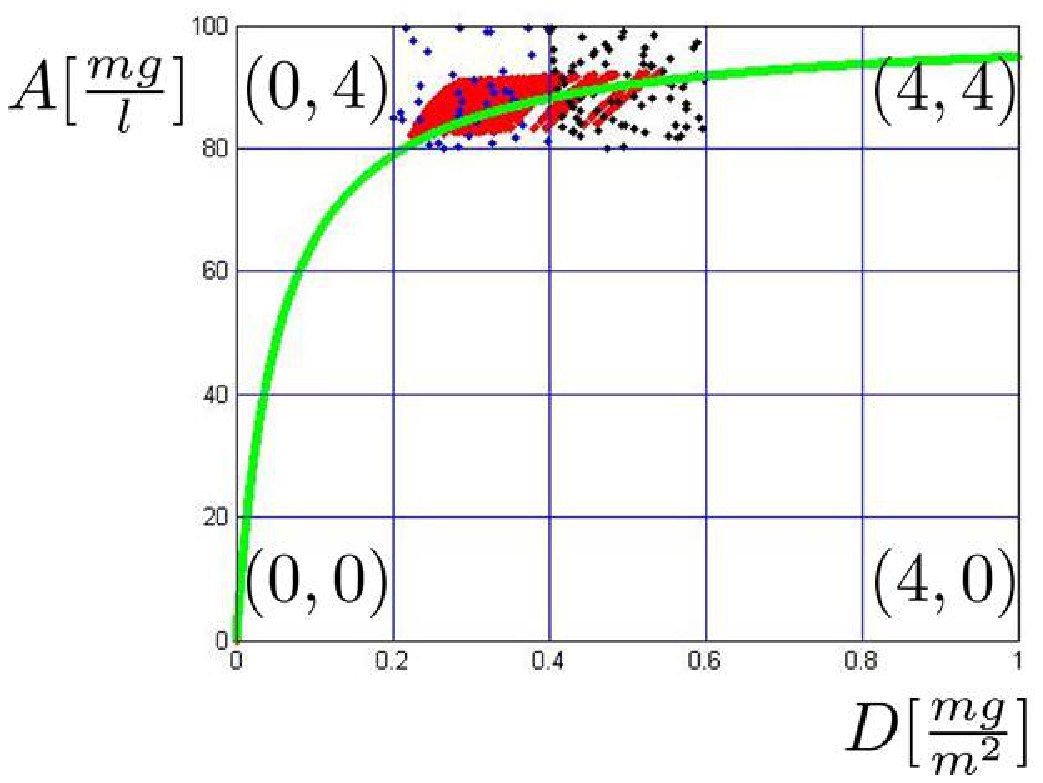}}
\hfill
\subfloat[\label{fig:C5P5_0h}]{\includegraphics[width=8cm]{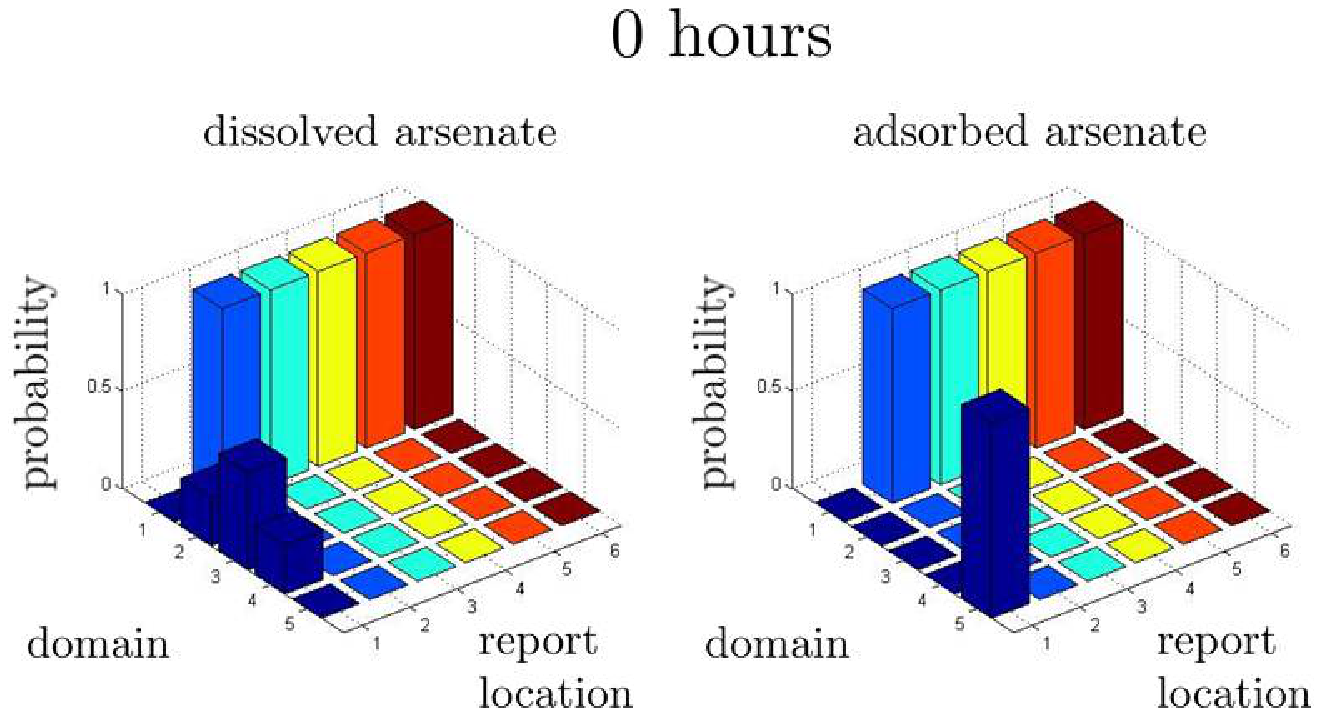}}

\subfloat[\label{fig:C5P5_1h}]{\includegraphics[width=8cm]{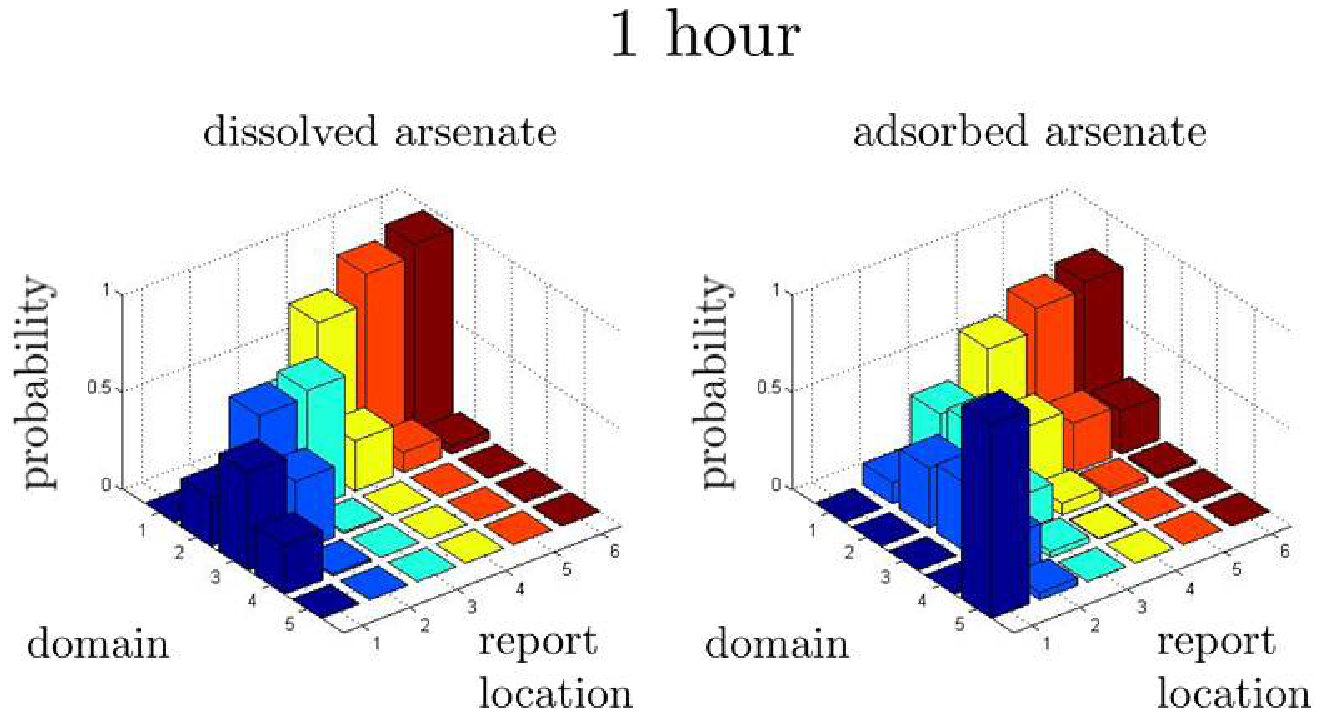}}
\hfill
\subfloat[\label{fig:C5P5_7h}]{\includegraphics[width=8cm]{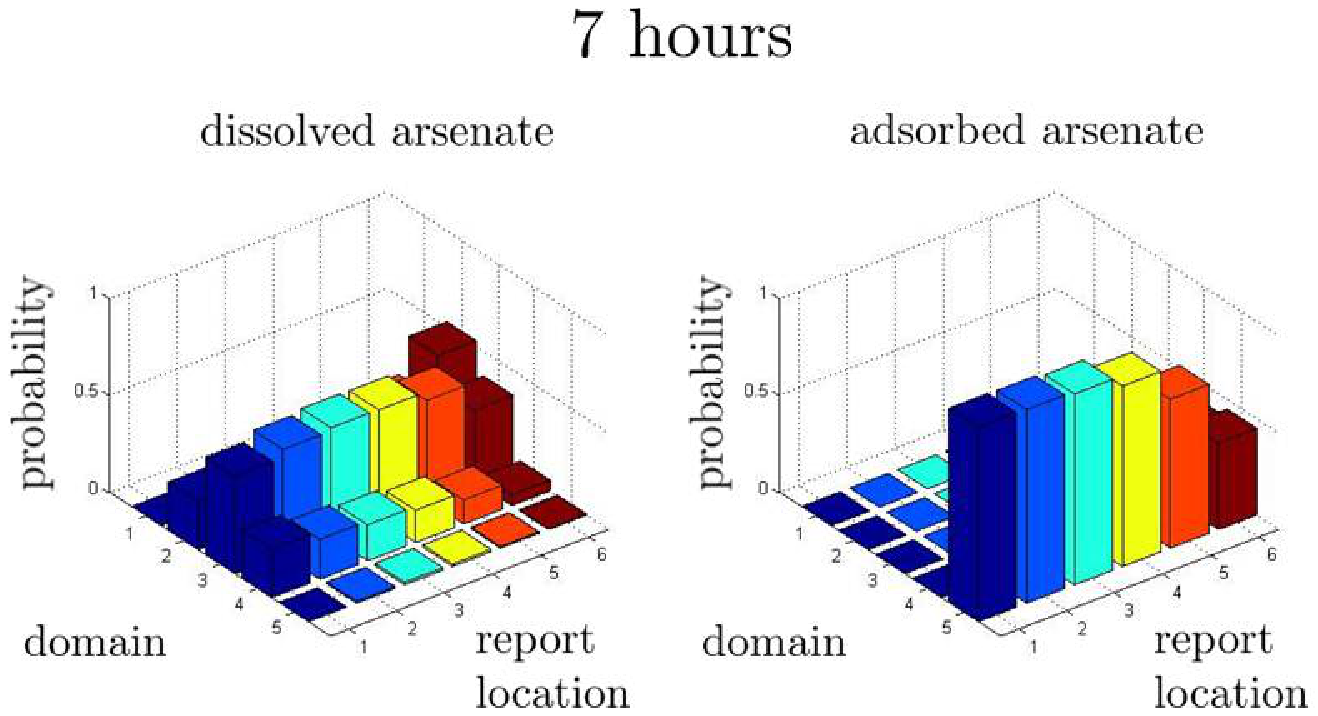}}
\caption{(a) A reservoir on a hill is connected to a consumer in a valley through a pipe with 6 report locations. (b) The phase space at every report location is divided into $5 \times 5$ coding domains, and the steady states are drawn in green. For example $f_0(((1,4), (2,4)))$ can be approximated by the transition of blue test points in domain $(1,4)$ and black ones in domain $(2,4)$ to the set of red points. (c) shows the initial conditions for an exemplary simulation with the according CPA, and the results after 1 and 7 hours are shown in (d) and (e), respectively. See also Fig. \ref{fig:steadystate}. }
\label{fig:setupdynamics}
\end{figure}

To obtain the local function of a CPA we map test points by using intermediate steps on the basis of the method of characteristics with the smaller $\Delta t'=0.1min$ and $\Delta x'=1m$. We use $V=\tilde V=\{0\}$ and partition the phase space equidistantly with 5 symbols in each of the $n=2$ directions. If we label the coding domains from $0$ to $4$ in each direction, the corresponding CPA results from transition probabilities like
$$f_0(((1,4),(2,4)))(\psi)= \left\{\begin{array}{ll}
0.806 & \mbox{ if } \psi=(1,4)\\
0.194 & \mbox{ if } \psi=(2,4)\\
0 & \mbox{ else }\\
\end{array}\right.,$$
see Fig. \ref{fig:phasespace_C5P5}. White noise boundary conditions are applied to describe a random arsenate source in the tank, and deterministic initial values represent a pipe which is completely empty in the beginning. The observed dynamics is shown in Fig. \ref{fig:C5P5_0h}-\ref{fig:C5P5_7h}: Dissolved arsenate is transported along the pipe, and over time the walls are covered more and more with adsorbed arsenate. After $24$ hours a steady state is reached, and we compare it to a Monte Carlo calculation, see Fig. \ref{fig:MC_C5P5}-\ref{fig:C5P5_24h}. The latter has also been obtained on the basis of the method of characteristics with $\Delta t'=0.1min$ and $\Delta x'=1m$ for $20000$ evaluations. The boundary condition has been drawn from the stationary boundary distribution every $10min$ and held constant in the meantime. Our example features an interesting probabilitic effect due to the nonlinearity of the reaction equations. Although the boundary values are distributed in the $D$-domains $2-4$, the consumer mostly observes dissolved arsenate at a concentration of domain $4$ in the steady state.

\begin{figure}
\subfloat[\label{fig:MC_C5P5}]{\includegraphics[width=8cm]{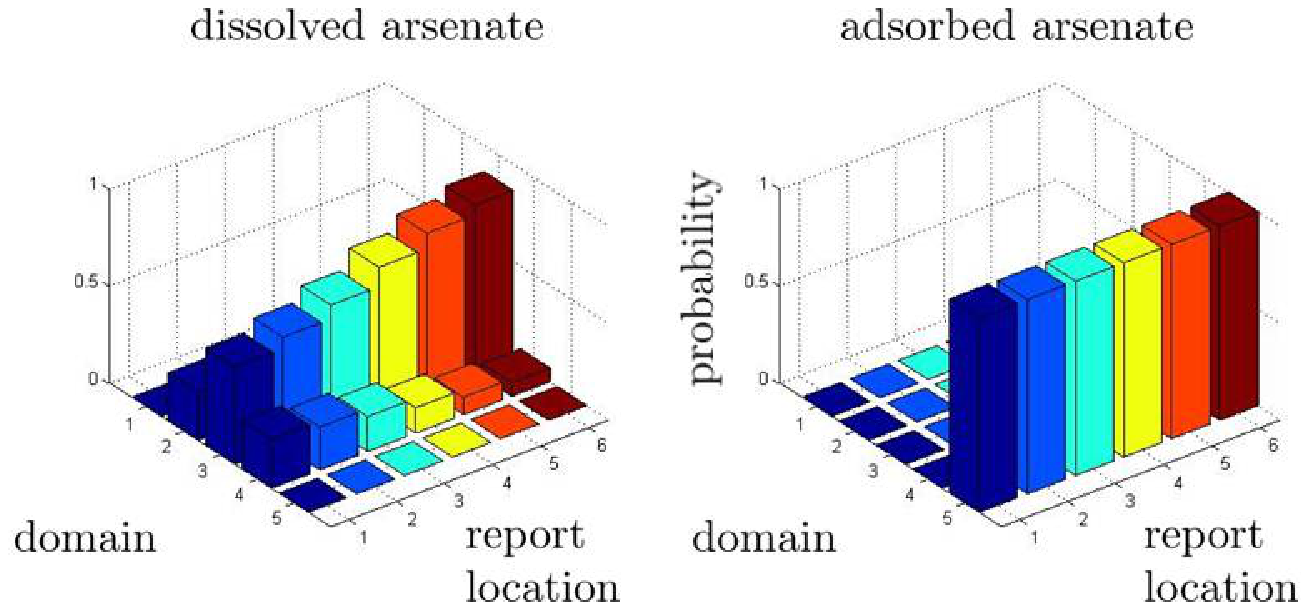}}
\hfill
\subfloat[\label{fig:C5P5_24h}]{\includegraphics[width=8cm]{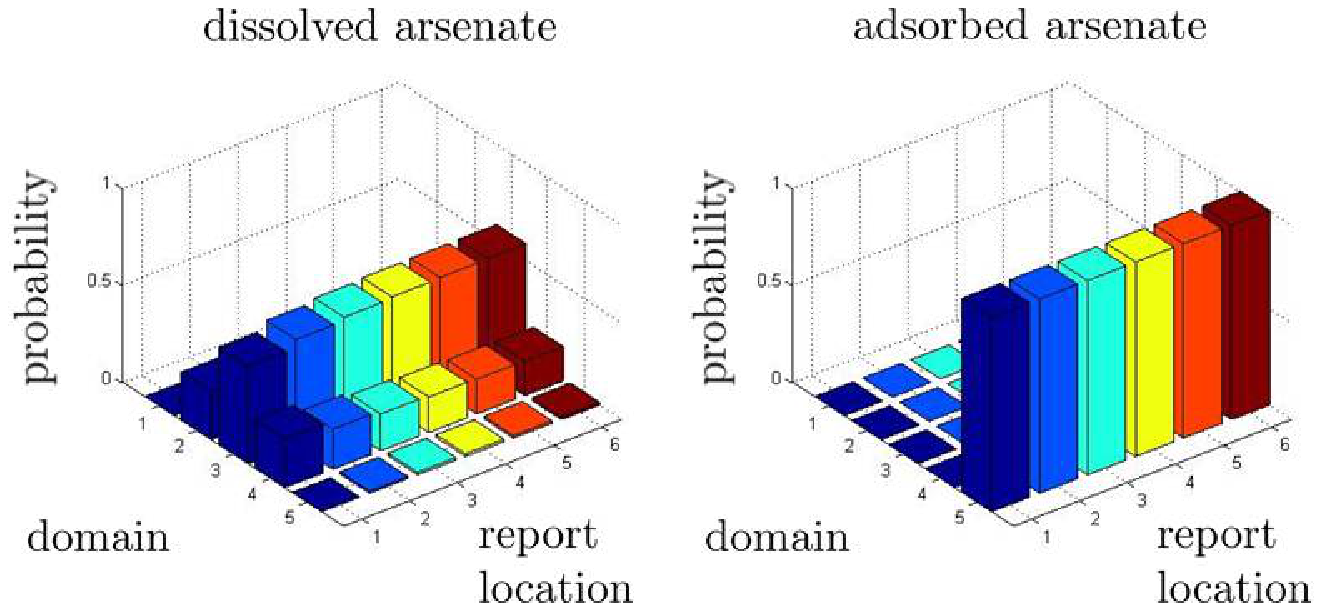}}

\subfloat[\label{fig:C5P15_24h}]{\includegraphics[width=8cm]{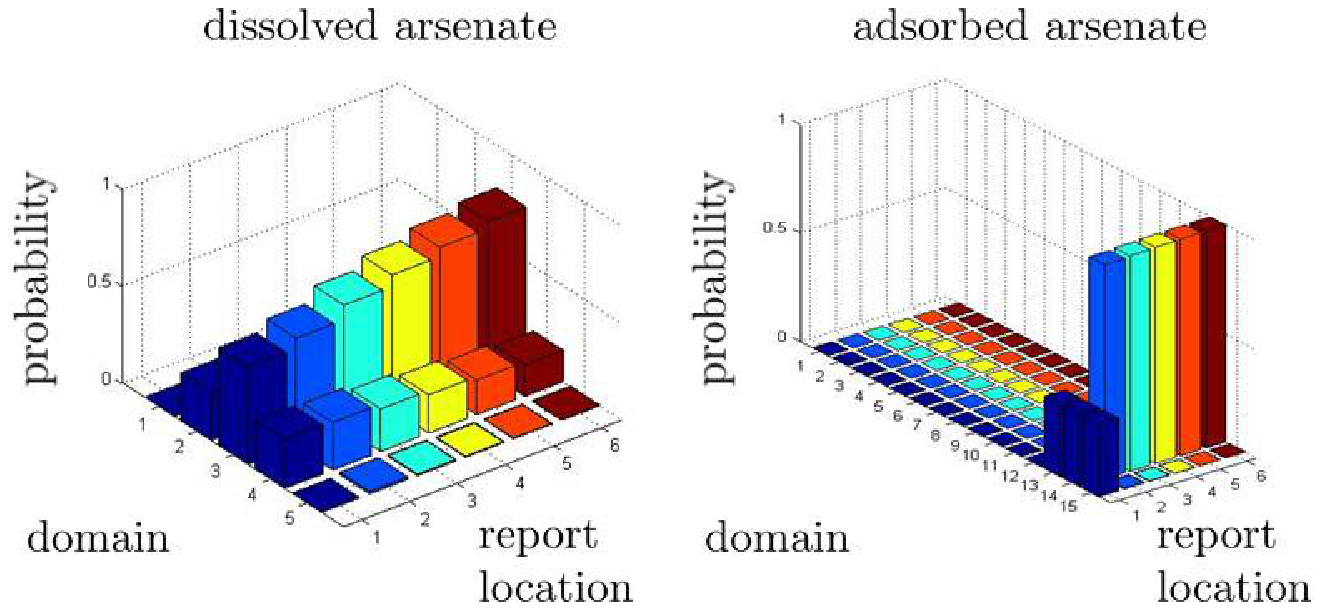}}
\hfill
\subfloat[\label{fig:C5P5_24h_W2}]{\includegraphics[width=8cm]{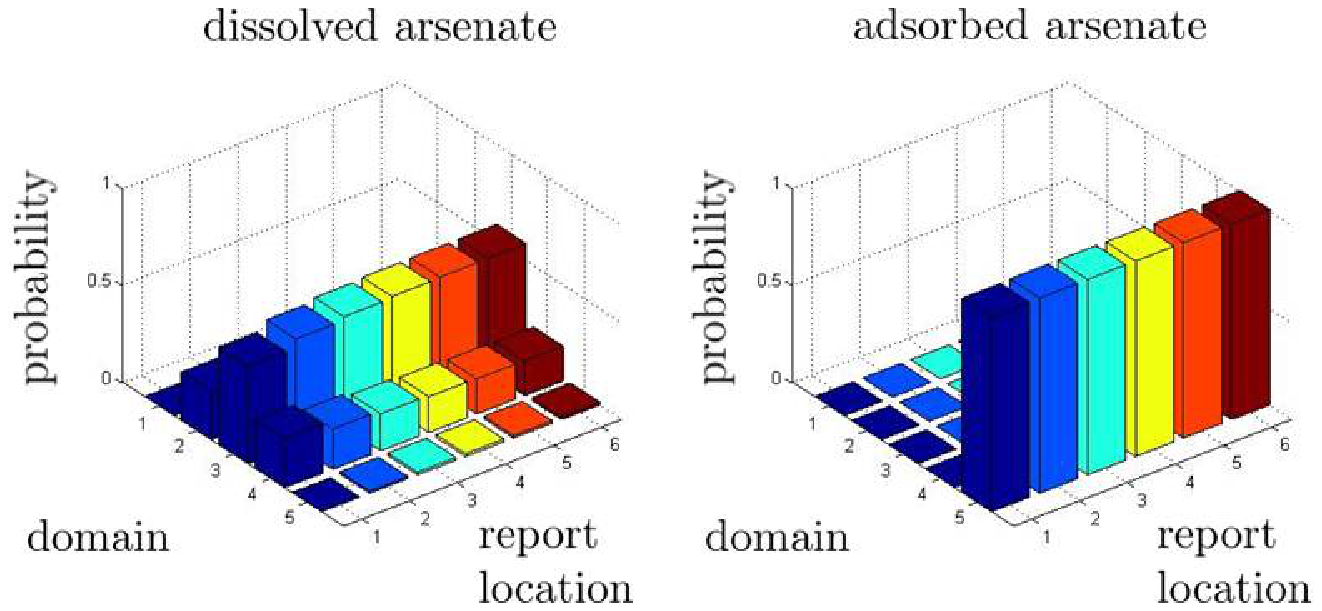}}

\caption{Steady states after $24$ hours. (a) shows the result of a Monte Carlo computation. In (b) one finds the result for a CPA with a state space resolution of $5\times 5$ and pattern size $\tilde V=V=\{0\}$, in (c) the results for $5 \times 15$ and $\tilde V=V=\{0\}$, and in (d) the ones for $5 \times 5$ and $\tilde V=\{0\}, V=\{-1,0\}$. }
\label{fig:steadystate}
\end{figure}

Furthermore we plot the steady state results from CPA for which the approximation parameters are altered. In Fig. \ref{fig:C5P15_24h} the result is plotted for $\tilde V=V=\{0\}$ with an equidistant phase space partition of $5$ domains in $D$- and $15$ domains in $A$-direction, whereas in Fig. \ref{fig:C5P5_24h_W2} the pattern size is extended by $W=\{-1,0\}$ to $V=\{-1,0\}$. It is observed that in this example increasing the pattern size does not improve the result if compared to the Monte Carlo case, but increasing the state space resolution has a notable effect. In all cases we used 75 test points for the coding domain at site $0$ and 37 at site $-1$ to approximate the local function, and a probability threshold of $0.00005$ in the simulation. 

We note that there is often no interest in global results and accordingly no need to transform between local and fully global states with $\hat \alpha$. Local information like indicated in the graphs can be directly extracted from the CPA result. Similarly, in practice the information about initial values is often given locally, such that there is no need to use the full $\hat \beta$. Besides, we note that the discrete state space information is often completely sufficient in practice. In the example a consumer is interested rather in risk level or threshold information about water contamination than in information in form of exact concentrations. In some biological sytems even the Boolean case, $|E|=2$, is enough \cite{Maria}.

\section{Conclusion}

\label{sec:Conclusions}

We have introduced a numerical scheme for density based uncertainty propagation in certain spatio-temporal systems. It consists of a preprocessing step, in which the underlying PDE system is translated into a cellular probabilistic automaton, and a simulation step, in which initial and boundary conditions are evolved. The simulation is parallelizable in the space extension and fast in relation to the preprocessing. Furthermore it computes on discrete states instead of on the continuous phase space. Because the discrete states can be interpreted as risk levels, fast uncertainty propagation directly on this simplified state space suites industrial demands. 

The algorithm is based on state space discretization like in set oriented numerics and on the de Bruijn state idea from cellular automata theory. There are two parameters that allow to control the approximation of the exact density evolution: state space resolution and de Bruijn pattern size. We have proven consistency of the method for uncertain initial conditions under deterministic dynamics and have paved the way towards the handling of spatio-temporal processes with more involved stochastic influence. More precisely, it has been shown how to deal with white noise boundary conditions, an important topic for example in contaminant transport modeling. However, it seems difficult to preserve temporal correlations in random parameters with our algorithm. Future research will focus on white noise parameters. We are confident that they only extend the preprocessing, whereas the simulation step is not changed. In this sense CPA promise to overcome the curse of dimension in parameter space. Besides, we want to investigate how our algorithm performs in more complex applications like the simulation of drinking or waste water grids.

\section{Acknowledgements}
The authors would like to thank Birgit Obst from Siemens Corporate Technology for helpful discussions about the arsenate example. 

\bibliographystyle{siam} 
\bibliography{dkohler_siam}

\end{document}